\documentclass[12pt,a4]{article}
\usepackage{amsmath, amsthm, amsfonts, amssymb, graphicx, color}
\numberwithin{equation}{section}
\newtheorem{theorem}{Theorem}[section]
\newtheorem{lemma}[theorem]{Lemma}

\newtheorem{proposition}[theorem]{Proposition}

\theoremstyle{definition}

\newtheorem{remark}[theorem]{Remark}
\newtheorem{definition}[theorem]{Definition}
\newtheorem{example}[theorem]{Example}

\theoremstyle{remark}

\newcommand{\R}{{\mathbb R}}

\newcommand{\msckw}{%
\footnotetext{\hspace{-0.75cm} 2010 {\it Mathematics Subject Classification}. 
{42B20, 42B25, 42B35, 46E30}
\endgraf
\hspace{-0.5cm} {\it Key words and phrases.} 
generalized fractional maximal function, generalized fractional integral, Orlicz spaces, generalized Orlicz-Morrey spaces
}
}

\begin{document}
\title{Generalized fractional maximal and integral operators on Orlicz
and generalized Orlicz--Morrey spaces of the third kind \msckw}
\author{Fatih~Deringoz,  Vagif~S.~Guliyev, Eiichi~Nakai, \\ Yoshihiro~Sawano and Minglei~Shi}
\date{}

\maketitle

\begin{abstract}
In the present paper, we will characterize the boundedness of the generalized fractional integral operators $I_{\rho}$
and
the generalized fractional maximal operators $M_{\rho}$ on Orlicz spaces, respectively.
Moreover, we will give a characterization for the Spanne-type boundedness
and
the Adams-type boundedness of the operators $M_{\rho}$ and $I_{\rho}$ on generalized Orlicz--Morrey spaces, respectively.
Also we give criteria for the weak versions of the Spanne-type boundedness
and
the Adams-type boundedness of the operators $M_{\rho}$ and $I_{\rho}$ on generalized Orlicz--Morrey spaces.
\end{abstract}

\section{Introduction}
\noindent

The aim of this paper is to obtain the necessary conditions
and the sufficient condtions
for
the generalized fractional maximal operator
$M_\rho$ and
the generalized fractional integral operator
$I_\rho$ to be bounded on
Orlicz spaces.
Our results can be extended
to
generalized Orlicz--Morrey spaces of the third kind
which will be defined later in this paper.

Let ${\mathbb R}^n$ be the $n$-dimensional Euclidean space.
For a function $\rho : (0,\infty) \to (0,\infty)$, let $I_{\rho}$
be the generalized fractional integral operator:
$$
I_{\rho}f(x) = \int_{\mathbb{R}^n }\frac{\rho(|x-y|)}{|x-y|^n} \, f(y) dy.
$$
Here $f$ is a suitable measurable function.
Note that this type of generalization goes back to \cite{Pustylnik99}.
If $\rho(r) = r^{\alpha}$, $0<\alpha<n$, then $I_{\rho}$ is the fractional integral operator or the Riesz potential and denoted by $I_{\alpha}$.
Hereafter, we assume that
\begin{equation}\label{int rho}
\int_{0}^{1}\frac{\rho(t)}{t}dt<\infty
\end{equation}
so that the fractional integrals $I_{\rho}f$ are well defined, at least for characteristic functions of balls.
The operator $I_{\rho}$ was introduced in \cite{Nakai2001Taiwan}
and
some partial results were announced in
\cite{Nakai2000ISAAC}.
We refer to \cite{Mizuta-Nakai-Ohno-Shimomura2010JMSJ} for the boundedness of $I_{\rho}$
on Orlicz space $L^{\Phi}(\Omega)$ with bounded domain $\Omega\subset\R^n$.
See also \cite{Nakai2001SCMJ,Nakai2002Lund,Nakai2004KIT,Nakai-Sumitomo2001SCMJ}
for the boundedness of $I_\rho$ on various spaces.
In these papers we assumed that
$\rho$ satisfies the doubling condition:
\begin{equation}\label{rho doubl}
 \frac{1}{C_1}\le\frac{\rho(r)}{\rho(s)}\le C_1,
 \quad\text{if} \ \ \frac12\le\frac rs\le2,
\end{equation}
and that $r\mapsto\rho(r)/r^n$ is almost decreasing:
\begin{equation}\label{rho/r^n}
 \frac{\rho(s)}{s^n}\le C_2\frac{\rho(r)}{r^n},
 \quad\text{if} \ \ r<s,
\end{equation}
where $C_1$ and $C_2$ are positive constants independent of $r,s\in(0,\infty)$.
Under these conditions we proved the boundedness of $I_{\rho}$ on Orlicz spaces in \cite{Nakai2000ISAAC,Nakai2001Taiwan}.

In this paper, instead of these conditions, we assume that
there exist positive constants $C$, $k_1$ and $k_2$ with $k_1<k_2$ such that, for all $r>0$,
\begin{equation}\label{sup rho}
 \sup_{r/2\le t\le r}\rho(t)
 \le
 C\int_{k_1r}^{k_2r}\frac{\rho(t)}{t}\,dt=:\tilde{\rho}(r).
\end{equation}
The condition \eqref{sup rho} was considered in \cite{Perez1994}
and also used in \cite{SawSugTan}. 
If $\rho$ satisfies \eqref{rho doubl} or \eqref{rho/r^n},
then $\rho$ satisfies \eqref{sup rho}.
Let
\begin{equation}\label{rho exmp}
 \rho(r)=
\begin{cases}
 r^{n}(\log(e/r))^{-1/2}, & 0<r<1,\\
 e^{-(r-1)}, & 1\le r<\infty.
\end{cases}
\end{equation}
Then $\rho$ satisfies \eqref{int rho} and \eqref{sup rho},
but fails \eqref{rho doubl} and \eqref{rho/r^n}.
Therefore, the results in this paper improve ones in \cite{Nakai2001Taiwan}.
Moreover, we give necessary and sufficient conditions for the boundedness of $I_{\rho}$
not only on Orlicz spaces but also on Orlicz--Morrey spaces
of the third kind.

Next, we define the generalized fractional maximal operator $M_{\rho}$.
For a function $\rho : (0,\infty) \to (0,\infty)$, let
\begin{equation}\label{Mr}
 M_{\rho} f(x) = \sup\limits_{r>0} \frac{\rho(r)}{|B(x,r)|} \int_{B(x,r)} |f(y)| dy,
\end{equation}
where $|G|$ is the Lebesgue measure of a measurable set $G \subset \R^n$.
We do not assume \eqref{int rho} on the function $\rho$ in \eqref{Mr}.
Instead we suppose that
$\rho$ is an increasing function such that
$r \in (0,\infty) \mapsto r^{-n}\rho(r) \in (0,\infty)$
is decreasing.

If $\rho\equiv1$, then $M_{\rho}$ is the Hardy-Littlewood maximal operator denoted by $M$.
If $\rho(r)=r^{\alpha}$, then $M_{\rho}$ is the usual fractional maximal operator denoted by $M_{\alpha}$.
We give
some necessary conditions and
some sufficient conditions for the boundedness of $M_{\rho}$
on Orlicz and Orlicz--Morrey spaces.


The structure of the remaining part of the present paper is as follows:
First we recall Young functions and Orlicz spaces in Section~\ref{sec:Young}.
In Section~\ref{sec:Ir Orl}, we
investigate the boundedness of generalized fractional integrals on Orlicz spaces.
We will give a necessary and sufficient condition for the boundedness of the generalized
fractional maximal operators on Orlicz spaces in Section~\ref{sec:Mr Orl}.
In Section~\ref{sec:Orl-Mor} we
discuss some properties of generalized Orlicz--Morrey spaces of the third kind.
Moreover, we will give necessary and sufficient conditions
for the Spanne and Adams-type boundedness of the generalized fractional integral operators
on generalized Orlicz--Morrey spaces of the third kind in Section~\ref{sec:Ir Orl-Mor}.
Finally, in Section~\ref{sec:Mr Orl-Mor}
we give criteria for the boundedness of the generalized fractional maximal operators
on generalized Orlicz--Morrey spaces of the third kind.

\section{Young functions and Orlicz spaces}\label{sec:Young}

We recall the definition of Young functions.
\begin{definition}\label{def2} A function $\Phi : [0,\infty] \rightarrow [0,\infty]$ is called a Young function 
if $\Phi$ is convex, left-continuous, $\lim\limits_{r\rightarrow +0} \Phi(r) = \Phi(0) = 0$ 
and $\lim\limits_{r\rightarrow \infty} \Phi(r) = \Phi(\infty)=\infty$.
\end{definition}
From the non-negativity, convexity and $\Phi(0) = 0$ it follows that any Young function is increasing.
We denote by $\mathcal{Y}$
the set of all Young  functions such that
\begin{equation*}
0<\Phi(r)<\infty \qquad \text{for} \qquad 0<r<\infty.
\end{equation*}
If $\Phi \in  \mathcal{Y}$, then $\Phi$ is absolutely continuous on every compact interval in $[0,\infty )$
and bijective from $[0,\infty )$ to itself.

Next we recall the generalized inverse of Young function $\Phi$
in the sense of O'Neil \cite[Definition~1.2]{ONeil1965}.
For a Young function $\Phi$ and  $0 \leq s \leq \infty $, let
\begin{equation*}
 \Phi^{-1}(s)=\inf\{r\geq 0: \Phi(r)>s\},
 \quad
 \text{where} \ \inf\emptyset=\infty.
\end{equation*}
Note that if $s<\infty$, then so is $\Phi^{-1}(s)$.
As in \cite[p. 301, Remarks]{ONeil1965},
we always have $\Phi^{-1}(\infty)=\infty$.
An important inequality we use is
\[
\Phi(\Phi^{-1}(r)) \le r \le \Phi^{-1}(\Phi(r)).
\]
See \cite[Property 1.3]{ONeil1965}.
Then
$\Phi^{-1}(s)$ is finite for all $s\in[0,\infty)$,
continuous on $(0,\infty)$
and right continuous at $s=0$.
Observe that  $\Phi^{-1}(\Phi(r)) = r$
if  $0 < \Phi(r) < \infty$
and that $\Phi(\Phi^{-1}(s))=s$
if $s=\Phi(r) \in (0,\infty)$ if $s \in [0,\Phi(\inf\{r>0\,:\,\Phi(r)=\infty\})]$.
Furthermore,
if $\Phi \in  \mathcal{Y}$,
then $\Phi^{-1}$ is the usual inverse function of $\Phi$.

\begin{remark}\label{dfghj}
For a Young function $\Phi$,
its inverse function $\Phi^{-1}$ is increasing and concave.
Hence, we have the following properties:
$$
\left\{
\begin{array}{ccc}
\Phi^{-1}(t)\ge \Phi^{-1}(\alpha t)\geq \alpha \Phi^{-1}(t),
& \text{  if } &0<\alpha<1 \\
\Phi^{-1}(t) \le \Phi^{-1}(\alpha t)\leq \alpha \Phi^{-1}(t),&\text{  if }& \alpha>1.
\end{array}
\right.
$$
Since $\Phi^{-1}$ is increasing,
the left inequality is clear.
In particular, $\Phi$ satisfies the doubling condition:
$\Phi^{-1}(2s)\le2\Phi^{-1}(s)$ for all $s\ge0$.

In fact for $0<\alpha<1$,
\begin{align*}
\Phi^{-1}(\alpha t)
&=
\inf\{s \ge 0\,:\,\Phi(s)>\alpha t\}.
\end{align*}
Since
$\displaystyle \dfrac{1}{\alpha}\Phi(s) \le \Phi\left(\dfrac{s}{\alpha}\right)$,
we have
\[
\Phi^{-1}(\alpha t)
\ge
\inf\left\{s \ge 0\,:\,\Phi\left(\dfrac{s}{\alpha}\right)>\alpha t\right\}
=
\alpha\inf\{s \ge 0\,:\,\Phi(s)> t\}=\alpha \Phi^{-1}(t).
\]
The right inequality for $\alpha>1$ is a consequence
of the one for $0<\alpha<1$.
\end{remark}

As in \cite[Property 1.6]{ONeil1965},
we have
\begin{equation}\label{2.3}
r\leq \Phi^{-1}(r)\widetilde{\Phi}^{-1}(r)\leq 2r \qquad \text{for } r\geq 0,
\end{equation}
where $\widetilde{\Phi}(r)$ is the complementary function of $\Phi$ defined by
\begin{equation*}
\widetilde{\Phi}(r)=\left\{
\begin{array}{ccc}
\sup\{rs-\Phi(s): s\in  [0,\infty )\}
& , & r\in  [0,\infty ) \\
\infty &,& r=\infty .
\end{array}
\right.
\end{equation*}
Then $\widetilde\Phi$ is also a Young function and $\widetilde{\widetilde\Phi}=\Phi$.

A Young function $\Phi$ is said to satisfy the
 $\Delta_2$-condition, denoted also by $\Phi \in  \Delta_2$, if
$$
\Phi(2r)\le C\Phi(r), \qquad r>0
$$
for some $C\ge1$. If $\Phi \in  \Delta_2$, then $\Phi \in  \mathcal{Y}$. 
A Young function $\Phi$ is said to satisfy the $\nabla_2$-condition, denoted also by  $\Phi \in  \nabla_2$, if
$$
\Phi(r)\leq \frac{1}{2C}\Phi(Cr),\qquad r\geq 0
$$
for some $C>1$.

We denote by $\chi_{G}$
the characteristic function of the set $G\subset{\mathbb R}^n$.

\begin{definition}[Orlicz Space]
For a Young function $\Phi$, the Orlicz space $L^{\Phi}({\mathbb R}^n)$ is defined by:
$$L^{\Phi}({\mathbb R}^n)=\left\{f\in  L^1_{\rm loc}({\mathbb R}^n): \int _{{\mathbb R}^n}\Phi(k|f(x)|)dx<\infty
 \text{ for some $k>0$  }\right\}.$$
The  space $L^{\Phi}_{\rm loc}({\mathbb R}^n)$ is defined as
the set of all measurable functions $f$ such that  $f\chi_{_B}\in  L^{\Phi}({\mathbb R}^n)$ for all balls $B \subset {\mathbb R}^n$.
\end{definition}

If $\Phi$ is a Young function,
then $L^{\Phi}({\mathbb R}^n)$ is a Banach space under the Luxemburg-Nakano norm
$$\|f\|_{L^{\Phi}}=\inf\left\{\lambda>0:\int _{{\mathbb R}^n}\Phi\Big(\frac{|f(x)|}{\lambda}\Big)dx\leq 1\right\}.$$
For example,
if $\Phi(r)=r^{p},\, 1\le p<\infty $, then $L^{\Phi}({\mathbb R}^n)=L^{p}({\mathbb R}^n)$.
If $\Phi(r)=0,\,(0\le r\le 1)$ and $\Phi(r)=\infty ,\,(r> 1)$, then $L^{\Phi}({\mathbb R}^n)=L^\infty ({\mathbb R}^n)$.

For a measurable set $\Omega\subset \mathbb{R}^{n}$, a measurable function $f$ and $t>0$, let
$$
m(\Omega,\ f,\ t)=|\{x\in\Omega:|f(x)|>t\}|.
$$
In the case $\Omega=\mathbb{R}^{n}$, we abbreviate it to $m(f,\ t)$.

Let $L^0({\mathbb R}^n)$ be the set of all measurable functions.
\begin{definition} For a Young function $\Phi$, the weak Orlicz space
$$
{\rm W}L^{\Phi}(\mathbb{R}^{n})=\{f\in L^{0}(\mathbb{R}^{n}):\Vert f\Vert_{{\rm W}L^{\Phi}}<\infty\}
$$
is defined by the quasi-norm
$$
\Vert f\Vert_{{\rm W}L^{\Phi}}=\sup_{\lambda>0} \| \lambda\, \chi_{_{(\lambda,\infty)}}(|f|)\|_{L^{\Phi}}.
$$
\end{definition}

For $\Omega\subset{\mathbb R}^n$, let
$$
\|f\|_{L^{\Phi}(\Omega)}:=\|f\chi_{\Omega}\|_{L^{\Phi}}
\quad\text{and}\quad
\|f\|_{{\rm W}L^{\Phi}(\Omega)}:=\|f\chi_{\Omega}\|_{{\rm W}L^{\Phi}}.
$$
A tacit understanding is that $f$ is defined to be zero outside $\Omega$.

We note that $\Vert f\Vert_{{\rm W}L^{\Phi}(\Omega)}\leq \Vert f\Vert_{L^{\Phi}(\Omega)}$, that
\begin{align*}
\|f\|_{{\rm W}L^\Phi}
&=\sup_{t>0}\Phi(t)m(\Omega,\ f,\ t)\\ 
&=\sup_{t>0}t\,m(\Omega,\ f,\ \Phi^{-1}(t))\\
&= \sup_{t>0}t\,m(\Omega,\ \Phi(|f|),\ t)
\end{align*}
and that
\begin{equation}\label{orlpr}
\int _{\Omega}\Phi\Big(\frac{|f(x)|}{\|f\|_{L^{\Phi}(\Omega)}}\Big)dx\leq 1,
\quad
\sup_{t>0}\Phi(t)m\Big(\Omega,\ \frac{f}{\|f\|_{{\rm W}L^{\Phi}(\Omega)}},\ t\Big)\leq 1
\end{equation}
according to \cite[Proposition 4.2]{KaNa18}.

The following analogue of the H\"older inequality is well known;
see 
\cite{Weiss50}
as well as
the paper \cite[\S II]{ONeil1965}
and the textbooks
\cite{KokKrbec,Sawano18}.
\begin{theorem}\label{HolderOr}
Let $\Omega\subset{\mathbb R}^n$ be a measurable set and,
let $f$ and $g$ be
measurable functions on $\Omega$. For a Young function $\Phi$ and its complementary function  $\widetilde{\Phi}$,
the following inequality is valid:
$$\int_{\Omega}|f(x)g(x)|dx \leq 2 \|f\|_{L^{\Phi}(\Omega)} \|g\|_{L^{\widetilde{\Phi}}(\Omega)}.$$
\end{theorem}

By elementary calculations we have the following property:
\begin{lemma}\label{charorlc}
Let $\Phi$ be a Young function and let $B$ be a set in $\mathbb{R}^n$ with finite Lebesgue measure. Then
\begin{equation*}
\|\chi_{_B}\|_{L^{\Phi}} = \|\chi_{_B}\|_{{\rm W}L^{\Phi}}=\frac{1}{\Phi^{-1}\left(|B|^{-1}\right)}.
\end{equation*}
\end{lemma}
By Theorem \ref{HolderOr}, Lemma \ref{charorlc} and \eqref{2.3} we get the following estimate:
\begin{lemma}\label{lemHold} 
For a Young function $\Phi$ and $B=B(x,r)$, the following inequality is valid:
$$\int_{B}|f(y)|dy \leq 2 |B| \Phi^{-1}\left(|B|^{-1}\right) \|f\|_{L^{\Phi}(B)}.$$
\end{lemma}

We recall the boundedness property
of the Hardy-Littlewood maximal operator $M$
on Orlicz spaces since we use it later.
\begin{theorem}\label{Maxorl}
Let $\Phi$ be a Young function.
\begin{enumerate}
\item
{\rm\cite[Theorem 1]{Cianchi99}}
The operator $M$ is bounded from $L^{\Phi}({\mathbb R}^n)$ to  ${\rm W}L^{\Phi}({\mathbb R}^n)$, and the inequality
\begin{equation}\label{MbdninqW}
\|M f\|_{{\rm W}L^{\Phi}}\leq C_0\|f\|_{L^{\Phi}}
\end{equation}
holds with constant $C_0$ independent of $f$.
\item
{\rm\cite[Theorem 1]{Cianchi99}},
{\rm\cite[Corollary 3.3]{Hasto15}}
The operator $M$ is bounded on $L^{\Phi}({\mathbb R}^n)$, and the inequality
\begin{equation}\label{MbdninqS}
\|M f\|_{L^{\Phi}}\leq C_0\|f\|_{L^{\Phi}}
\end{equation}
holds with constant $C_0$ independent of $f$ if and only if $\Phi\in\nabla_2$.
\end{enumerate}
\end{theorem}
See the textbooks \cite{KokKrbec,Maligranda89,RaRe91,Sawano18} for 
more about Orlicz spaces.

\section{Generalized fractional integrals on Orlicz spaces}\label{sec:Ir Orl}

The following theorem is one of our main results and gives necessary and sufficient conditions for the
boundedness of the operator $I_{\rho}$
from $L^{\Phi}({\mathbb R}^n)$ to ${\rm W}L^{\Psi}({\mathbb R}^n)$ and from $L^{\Phi}({\mathbb R}^n)$ to $L^{\Psi}({\mathbb R}^n)$.

\begin{theorem}  \label{AdGulGenRszOrlMorNec}
Let $\Phi,\Psi$ be Young functions.
\begin{enumerate}
\item
Let $\rho$ satisfy the conditions \eqref{int rho} and \eqref{sup rho}.
Then the condition
\begin{equation}\label{adRieszCharOrl1}
\Phi^{-1}(r^{-n}) \int_{0}^{r}\frac{\rho(t)}{t}dt + \int_{r}^{\infty}\rho(t) \, \Phi^{-1}(t^{-n}) \frac{dt}{t}  \le C \Psi^{-1}\big(r^{-n}\big)
\end{equation}
for all $r>0$, where $C>0$ does not depend on $r$, is sufficient for the boundedness of $I_{\rho}$ from $L^{\Phi}({\mathbb R}^n)$ to ${\rm W}L^{\Psi}({\mathbb R}^n)$.
Moreover, if $\Phi\in\nabla_2$, then the condition \eqref{adRieszCharOrl1} is also sufficient for the boundedness of $I_{\rho}$ from $L^{\Phi}({\mathbb R}^n)$ to $L^{\Psi}({\mathbb R}^n)$.

\item
The condition
\begin{equation}\label{adRieszCharOrl2}
\Phi^{-1}(r^{-n}) \int_{0}^{r}\frac{\rho(t)}{t}dt\le C \Psi^{-1}(r^{-n})
\end{equation}
for all $r>0$, where $C>0$ does not depend on $r$, is necessary for the boundedness of $I_{\rho}$ from $L^{\Phi}({\mathbb R}^n)$ to ${\rm W}L^{\Psi}({\mathbb R}^n)$ and from $L^{\Phi}({\mathbb R}^n)$ to $L^{\Psi}({\mathbb R}^n)$.

\item
Let $\rho$ satisfy the conditions \eqref{int rho} and \eqref{sup rho}.
Assume the condition
\begin{equation}\label{adRieszCharOrl3}
\int_{r}^{\infty}\rho(t) \, \Phi^{-1}(t^{-n}) \frac{dt}{t} \le C \Psi^{-1}\big(r^{-n}\big)
\end{equation}
holds for all $r>0$, where $C>0$ does not depend on $r$.
Then
condition \eqref{adRieszCharOrl2}
is necessary and sufficient for the boundedness of $I_{\rho}$ from $L^{\Phi}({\mathbb R}^n)$ to ${\rm W}L^{\Psi}({\mathbb R}^n)$.
Moreover, if $\Phi\in\nabla_2$, then the condition \eqref{adRieszCharOrl2} is necessary and sufficient for the boundedness of $I_{\rho}$ from $L^{\Phi}({\mathbb R}^n)$ to $L^{\Psi}({\mathbb R}^n)$.
\end{enumerate}
\end{theorem}

\begin{remark}\label{rem:LP}
We cannot replace $\int_0^r\frac{\rho(t)}{t}\,dt$ by $\rho(r)$ in \eqref{adRieszCharOrl1},
see
\cite[Section~5]{Nakai-Sumitomo2001SCMJ}.
\end{remark}

We need a couple of auxilary estimates.
The following lemma was proved in \cite[Lemma 2.1]{ErGunNSaw}:
\begin{lemma}\label{pwsgenfr}
There exist a constant $C>0$ such that for all $x\in B(0,r/2)$ and $r>0$,
$$
\int_{0}^{r/2}\frac{\rho(t)}{t}dt\le C I_{\rho}\chi_{B(0,r)}(x)
$$
holds.
\end{lemma}

\begin{proposition}\label{prop:180314-1}
Let $\rho$ satisfy \eqref{sup rho}.
Define
\begin{equation}\label{eq:tilde rho}
\tilde{\rho}(r)=\int_{k_1 r}^{k_2 r} \rho(s) \frac{ds}{s}
\quad (r>0).
\end{equation}
Let $\tau:(0,\infty) \to (0,\infty)$ be a doubling function
in the sense that $\tau(r) \sim \tau(s)$ if $0<s \le r \le 2s$.
Then, for each $r>0$,
\begin{align}
\label{eq:180314-9}
  \sum_{j=-\infty}^{-1}\tilde{\rho}(2^j r) &\lesssim
 \int_{0}^{k_2 r}\frac{\rho(s)}{s}ds,\\
\label{eq:180314-10}
 \sum_{j=0}^{\infty}\tilde{\rho}(2^j r)\tau\big((2^{j}r)^{-n}\big)
&\lesssim \int_{k_1 r}^{\infty} \frac{\rho(s)}{s}\tau\big(s^{-n}\big) ds.
\end{align}
\end{proposition}

\begin{proof}
We invoke the overlapping property in \cite{SawSugTan} and by Remark \ref{dfghj} we have
\begin{align*}
  \sum_{j=-\infty}^{-1}\tilde{\rho}(2^j r)  & = \sum_{j=-\infty}^{-1}\int_{2^{j}k_1 r}^{2^{j}k_2 r} \rho(s) \frac{ds}{s}\\
&\le \int_{0}^{k_2 r} \left(\sum_{j=-\infty}^{-1}\chi_{[2^{j}k_1r,~2^{j}k_2 r]}(s)\right) \frac{\rho(s)}{s}ds\\
&\lesssim \int_{0}^{k_2 r}\frac{\rho(s)}{s}ds
\end{align*}
and
\begin{align*}
  \sum_{j=0}^{\infty}\tilde{\rho}(2^j r)\tau\big((2^{j}r)^{-n}\big) &= \int_{k_1 r}^{\infty} \left(\sum_{j=0}^{\infty}\chi_{[2^{j}k_1r,~2^{j}k_2 r]}(s)\frac{\rho(s)}{s}\tau\big((2^{j}r)^{-n}\big)\right) ds\\
  & \lesssim \int_{k_1 r}^{\infty}\left(\sum_{j=0}^{\infty}\chi_{[2^{j}k_1r,~2^{j}k_2 r]}(s)\right)\frac{\rho(s)}{s}\tau\big(s^{-n}\big) ds\\
  & \lesssim \int_{k_1 r}^{\infty} \frac{\rho(s)}{s}\tau\big(s^{-n}\big) ds.
\end{align*}
\end{proof}

To prove Theorem \ref{AdGulGenRszOrlMorNec},
we need the following estimate of Hedberg-type \cite{Hedberg}:
\begin{proposition}\label{propoverlap}
Under the assumption of Theorem \ref{AdGulGenRszOrlMorNec},
for any positive constant $C_0$, there exists a positive constant $C_1$ such that,
for all nonnegative functions $f\in L^{\Phi}({\mathbb R}^n)$ with $f\ne0$,
\begin{equation}\label{pwerfdsc}
I_{\rho} f(x)  \leq
C_1
\|f\|_{L^{\Phi}} \Psi^{-1}\circ\Phi \Big(\frac{Mf(x)}{C_0\|f\|_{L^{\Phi}}}\Big)
\qquad (x\in{\mathbb R}^n).
\end{equation}
\end{proposition}

\begin{proof}
The idea of the proof comes from \cite{ErGunNSaw}. 
First note that
\[
0<\Phi^{-1}(0)\int_0^\infty \frac{\rho(t)}{t}\,dt \lesssim \Psi^{-1}(0)
\]
as long as $\Phi^{-1}(0)>0$.

Let $x\in{\mathbb R}^n$.
Keeping in mind that $M f(x)>0$,
we may assume 
\[
0<\frac{Mf(x)}{C_0\|f\|_{L^{\Phi}}}<\infty, \quad
0 \le \Phi\left(\frac{Mf(x)}{C_0\|f\|_{L^{\Phi}}}\right)<\infty;
\]
otherwise there is nothing to prove.
If
\[
\Phi\left(\frac{Mf(x)}{C_0\|f\|_{L^{\Phi}}}\right)=0,
\]
then
\[
\frac{Mf(x)}{C_0\|f\|_{L^{\Phi}}}
\le
\sup\{u \ge 0\,: \,\Phi(u)=0\} =\Phi^{-1}(0)
\]  
and hence
\begin{align*}
I_\rho f(x)
&\le C
\sum_{j=-\infty}^\infty
\frac{\tilde{\rho}(2^j)}{2^{j n}}
\int_{|x-y|<2^j}|f(y)|\,dy\\
&\le C
\left(\int_0^\infty \frac{\rho(s)}{s}\,ds\right)M f(x)\\
&\le C
\frac{\Psi^{-1}(0)}{\Phi^{-1}(0)}M f(x)\\
&\le C
\frac{1}{\Phi^{-1}(0)}
\Psi^{-1}\left(\Phi\left(\frac{Mf(x)}{C_0\|f\|_{L^{\Phi}}}\right)\right)M f(x)\\
&\le C
\Psi^{-1}\left(\Phi\left(\frac{Mf(x)}{C_0\|f\|_{L^{\Phi}}}\right)\right)
\|f\|_{L^\Phi}.
\end{align*}
So, this case the result is valid.

If
\[
\Phi\left(\frac{Mf(x)}{C_0\|f\|_{L^{\Phi}}}\right)>0,
\] 
choose $r \in (0,\infty)$ so that
\[
r^{-n}=\Phi\left(\frac{Mf(x)}{C_0\|f\|_{L^{\Phi}}}\right).
\]
We have
\begin{equation*}
I_{\rho} f(x)\le
C\left[\sum_{j=-\infty}^{-1}+\sum_{j=0}^{\infty}\frac{\tilde{\rho}(2^j r)}{(2^j r)^n}\int_{|x-y|<2^{j}r}f(y) dy\right]
=C({\rm I}+{\rm II})
\end{equation*}
for given $x\in {\mathbb R}^n$ and $r>0$.

Then
from Proposition \ref{prop:180314-1}
\begin{align*}
{\rm I}&\le C \sum_{j=-\infty}^{-1}\tilde{\rho}(2^j r) Mf(x)\le C \left(\int_{0}^{k_2 r}\frac{\rho(s)}{s}ds\right) Mf(x)\\
{\rm II} &\le C \sum_{j=0}^{\infty}\tilde{\rho}(2^j r)\Phi^{-1}\big((2^{j}r)^{-n}\big)\|f\|_{L^{\Phi}(B(x,2^{j}r))}\\
&\le C \|f\|_{L^{\Phi}} \int_{k_1 r}^{\infty} \Phi^{-1}\big(s^{-n}\big) \frac{\rho(s)}{s}ds.
\end{align*}
Consequently, we have
\begin{equation*}
\begin{split}
I_\rho f(x) \lesssim \left(\int_{0}^{k_2 r}\frac{\rho(s)}{s}ds\right) Mf(x)+\|f\|_{L^{\Phi}}\int_{k_1 r}^{\infty} \Phi^{-1}\big(s^{-n}\big) \frac{\rho(s)}{s}ds.
\end{split}
\end{equation*}
Thus, by the doubling property of $\Phi^{-1}$ and $\Psi^{-1}$,
\eqref{adRieszCharOrl1} and Remark \ref{dfghj} we obtain
\begin{align*}
I_{\rho} f(x)
&\lesssim
Mf(x)\frac{\Psi^{-1}((k_2 r)^{-n})}{\Phi^{-1}((k_2 r)^{-n})}
+
\|f\|_{L^{\Phi}} \, \Psi^{-1}((k_1 r)^{-n})\\
&\lesssim
Mf(x)\frac{\Psi^{-1}(r^{-n})}{\Phi^{-1}(r^{-n})}
+
\|f\|_{L^{\Phi}} \, \Psi^{-1}(r^{-n}).\\
\end{align*}
Recall that
$\Phi^{-1}(\Phi(r)) = r$
if  $0 < \Phi(r) < \infty$.
Thus
$\Phi^{-1}(r^{-n})=\dfrac{Mf(x)}{C_0\|f\|_{L^{\Phi}}}$ and
\begin{equation*}
 I_{\rho} f(x)
 \lesssim
 \|f\|_{L^{\Phi}} \, \Psi^{-1}(r^{-n})
 =
 \|f\|_{L^{\Phi}} \, \Psi^{-1}\left(\Phi\left(\frac{Mf(x)}{C_0\|f\|_{L^{\Phi}}}\right)\right).
\end{equation*}
Therefore, we get
(\ref{pwerfdsc}).
\end{proof}

Now we move on to the proof of Theorem~\ref{AdGulGenRszOrlMorNec}.
The third statement is a consequence
of the remaining statements.
So we concentrate on the first and the second ones.
\begin{itemize}
\item
Let $C_0$ be as in \eqref{MbdninqW}.
Let $f$ be a non-negative measurable function.
Then by \eqref{MbdninqW} and \eqref{pwerfdsc},
\begin{align*}
&\sup_{r>0}\Psi(r)\, m\Big(\frac{I_{\rho} f(x)}{C_1\|f\|_{L^{\Phi}}},r\Big)=\sup_{r>0}r\, m\Big(\Psi\Big(\frac{I_{\rho} f(x)}{C_1\|f\|_{L^{\Phi}}}\Big),r\Big)
\\
\leq &\sup_{r>0}r\, m\Big(\Phi\Big(\frac{M f(x)}{C_0\|f\|_{L^{\Phi}}}\Big),r\Big)\leq\sup_{r>0}\Phi(r)\, m\Big(\frac{M f(x)}{\|Mf\|_{{\rm W}L^{\Phi}}},r\Big)\leq 1,
\end{align*}
i.e.
$$
\|I_{\rho}f\|_{{\rm W}L^{\Psi}}\lesssim \|f\|_{L^{\Phi}}.
$$
\item
Assume in addition that $\Phi\in\nabla_2$,
so that we have \eqref{MbdninqS}. By \eqref{MbdninqS}, we have
\begin{align*}
\int_{{\mathbb R}^n}\Psi\left(\frac{I_{\rho} f(x)}{C_1\|f\|_{L^{\Phi}}}\right)dx&\leq \int_{{\mathbb R}^n}\Phi\left(\frac{M f(x)}{C_0\|f\|_{L^{\Phi}}}\right)dx\\
&\leq \int_{{\mathbb R}^n}\Phi\left(\frac{M f(x)}{\|Mf\|_{L^{\Phi}}}\right)dx\leq 1,
\end{align*}
i.e.
$$
\|I_{\rho}f\|_{L^{\Psi}}\lesssim \|f\|_{L^{\Phi}}.
$$
\item
We can and do concentrate on the boundedness
 of $I_{\rho}$ from $L^{\Phi}({\mathbb R}^n)$ to ${\rm W}L^{\Psi}({\mathbb R}^n)$,
since the  boundedness
 of $I_{\rho}$ from $L^{\Phi}({\mathbb R}^n)$ to $L^{\Psi}({\mathbb R}^n)$
is stronger than the boundedness
 of $I_{\rho}$ from $L^{\Phi}({\mathbb R}^n)$ to ${\rm W}L^{\Psi}({\mathbb R}^n)$.
With this in mind,
assume that $I_{\rho}$ is bounded from $L^{\Phi}({\mathbb R}^n)$ to ${\rm W}L^{\Psi}({\mathbb R}^n)$.

Then we have by Lemma~\ref{pwsgenfr}
\begin{equation*}
 \int_{0}^{r/2}\frac{\rho(s)}{s}ds \, \|\chi_{B(0,r/2)}\|_{{\rm W}L^{\Psi}(B(0,r/2))}
 \lesssim
 \|I_{\rho} \chi_{B(0,r)}\|_{{\rm W}L^{\Psi}(B(0,r/2))}.
\end{equation*}
Therefore, by the doubling property of $\Phi^{-1}$ and Lemma~\ref{charorlc}, we have
\begin{align*}
\int_{0}^{r/2}\frac{\rho(s)}{s}ds
&\lesssim \Psi^{-1}\left(r^{-n}\right)\|I_{\rho} \chi_{B(0,r)}\|_{{\rm W}L^{\Psi}(B(0,r/2))} \\
&\lesssim \Psi^{-1}\left(r^{-n}\right)\|I_{\rho} \chi_{B(0,r)}\|_{{\rm W}L^{\Psi}}
\\
&\lesssim \Psi^{-1}\left(r^{-n}\right)\|\chi_{B(0,r)}\|_{L^{\Phi}}\\
&\lesssim
\frac{\Psi^{-1}\left(r^{-n}\right)}{\Phi^{-1}\left(r^{-n}\right)}.
\end{align*}
\end{itemize}

\begin{remark}
In \cite[Corollary 3.2]{Nakai2001Taiwan}
the third author found the sufficient conditions which ensures
the boundedness of the operator $I_{\rho}$ from $L^{\Phi}({\mathbb R}^n)$ to $L^{\Psi}({\mathbb R}^n)$, including its weak version.
Theorem \ref{AdGulGenRszOrlMorNec} improves
the third author's result
in that Theorem \ref{AdGulGenRszOrlMorNec} also covers the necessity
by imposing a weaker condition on $\rho$.
\end{remark}

\begin{remark}
In the case $\Phi(t)=t^p$,
Theorem \ref{AdGulGenRszOrlMorNec} was proved in \cite[Corollary 1.5]{ErGunNSaw}.
\end{remark}

\begin{example}\label{exmp:Ir}
Let $\rho$ be as in \eqref{rho exmp} and
\begin{equation*}
 \Phi(t)=
 \begin{cases}
 t^{3/2}, & 0\le t\le1, \\
 t(\log(et))^{1/2}, & t>1,
 \end{cases}
 \quad
 \Psi(t)=
 \begin{cases}
 \dfrac{2e}{3}t^{3/2}, & 0\le t\le1, \\[1ex]
 \dfrac{2e}{3}\dfrac{\exp\exp(t)}{\exp\exp(1)}, & t>1.
 \end{cases}
\end{equation*}
Then the pair $(\rho,\Phi,\Psi)$ satisfies \eqref{adRieszCharOrl1}. 
In fact,
we have
\begin{equation*}
 \Phi^{-1}(u)\sim
 \begin{cases}
 u^{2/3}, & 0\le u\le1, \\
 u(\log(eu))^{-1/2}, & u>1,
 \end{cases}
\end{equation*}
\begin{equation*}
 \Psi^{-1}(u)\sim
 \begin{cases}
 u^{2/3}, & 0\le u\le1, \\
 \log(\log(e^e u)), & u>1,
 \end{cases}
\end{equation*}
and,
for all $r>0$,
\begin{gather*}
 \int_0^r\frac{\rho(t)}{t}\,dt\;{\Phi}^{-1}(1/r^n) 
 \lesssim
\min(1,r^{-2n/3}),
\\
 \int_r^{\infty}\frac{\rho(t)\,\Phi^{-1}(1/t^n)}{t}\,dt
 \lesssim
\min( \log\log(e^e/r),r^{-2n/3}),
\\
 \Psi^{-1}(1/r^n)
 \sim
\min( \log\log(e^e/r),r^{-2n/3}).
\end{gather*}
See \cite{Nakai2001SCMJ} for
other examples.
\end{example}

\section{Generalized fractional maximal operators on Orlicz spaces}\label{sec:Mr Orl}

We recall that,
for a function $\rho : (0,\infty) \to (0,\infty)$,
$M_{\rho}$ is defined by \eqref{Mr}.
Here we suppose that
$\rho$ is an increasing function such that
$r \in (0,\infty) \mapsto r^{-n}\rho(r) \in (0,\infty)$
is decreasing.

Under this assumption, we have the following localized estimate:
\begin{lemma}\label{swkshr}
There exists a positive constant $C$ such that,
for all balls $B=B(x,r)$
and all measurable functions $f$ supported on $B$,
\begin{equation}\label{eq:180314-2}
M_{\rho} f(x)\le C \rho(r) M f(x).
\end{equation}
\end{lemma}

\begin{proof}
Let $B(R)=B(x,R)$ with $x=0$ for $R>0$.
By the definition of $M_{\rho}$, we have
\begin{align*}
\lefteqn{
M_{\rho}f(x)
}\\
&=
\max\left\{
\sup\limits_{0<R<3r} \frac{\rho(R)}{|B(R)|} \int_{B(x,R)} |f(y)| dy,
\sup\limits_{R \ge 3r} \frac{\rho(R)}{|B(R)|} \int_{B(x,R)} |f(y)| dy
\right\}.
\end{align*}
For the first term,
we use the fact that $\rho$ is increasing and doubling to have
\begin{align*}
\sup\limits_{0<R<3r} \frac{\rho(R)}{|B(R)|} \int_{B(x,R)} |f(y)| dy
&\lesssim
\sup\limits_{0<R<3r} \frac{\rho(r)}{|B(x,R)|} \int_{B(x,R)} |f(y)| dy\\
&\le
\rho(r)M f(x).
\end{align*}
For the second term,
since $r \in (0,\infty) \mapsto r^{-n}\rho(r) \in (0,\infty)$ is decreasing and $f$ is supported on $B(x,r)$
\begin{align*}
\sup\limits_{R \ge 3r} \frac{\rho(R)}{|B(R)|} \int_{B(x,R)} |f(y)| dy
&\le
\sup\limits_{R \ge 3r} \frac{\rho(3r)}{|B(x,3r)|} \int_{B(x,R)} |f(y)| dy\\
&\le \rho(r)M f(x).
\end{align*}
Thus, combining these estimates, we obtain the desired result.
\end{proof}



The Hedberg inequality for $M_\rho$ and $L^\Phi$ can be stated as follows:
\begin{lemma}\label{lem:180314-5}
Let $\Phi, \Psi$ be Young functions.
Assume that there exists a positive constant $C$ such that, for all $r>0$,
\begin{equation}\label{adRieszCharOrl2GFM}
\rho(r)\le C \frac{\Psi^{-1}\big(r^{-n}\big)}{\Phi^{-1}\big(r^{-n}\big)}.
\end{equation}
Then, for any positive constant $C_0$, there exists a positive constant $C_1$ such that,
for all $f\in L^{\Phi}({\mathbb R}^n)$ with $f\ne0$,
\begin{equation}\label{eq:180314-7}
M_{\rho} f(x)
\le C_1
\|f\|_{L^{\Phi}} (\Psi^{-1}\circ\Phi) \Big(\frac{Mf(x)}{C_0\|f\|_{L^{\Phi}}}\Big) \quad (x \in {\mathbb R}^n).
\end{equation}
\end{lemma}

\begin{proof}
First note that
\begin{equation}\label{eq:181109-1}
\lim_{r \to \infty}\rho(r) \lesssim \frac{\Psi^{-1}(0)}{\Phi^{-1}(0)}
\end{equation}
if $\Phi^{-1}(0)>0$.
Let $x\in{\mathbb R}^n$ be an arbitrary point.
We may assume that $0<Mf(x)<\infty$
keeping in mind that $f$ does not vanish on a set of positive measure.
Furthermore,
we can assume that
\[
\Phi\Big(\frac{Mf(x)}{C_0\|f\|_{L^{\Phi}}}\Big)<\infty;
\]
otherwise there is nothing to do
since $\Psi^{-1}(\infty)=\infty$.
If
\[
\Phi\Big(\frac{Mf(x)}{C_0\|f\|_{L^{\Phi}}}\Big)=0,
\]
then
\[
\Phi^{-1}(0) \ge\frac{Mf(x)}{C_0\|f\|_{L^{\Phi}}}>0
\]
according to the definition of $\Phi^{-1}$.
Thus, thanks to (\ref{eq:181109-1})
\begin{align*}
\lim_{r \to \infty}\rho(r) 
&\lesssim \frac{\Psi^{-1}(0)}{\Phi^{-1}(0)}\\
&=\frac{1}{\Phi^{-1}(0)}\Psi^{-1}\circ\Phi\left(\frac{Mf(x)}{C_0\|f\|_{L^{\Phi}}}\right)\\
&\le
\frac{C_0\|f\|_{L^{\Phi}}}{Mf(x)}\Psi^{-1}\circ\Phi\left(\frac{Mf(x)}{C_0\|f\|_{L^{\Phi}}}\right).
\end{align*}
Thus by
(\ref{eq:180314-2})
we have
\[
M_\rho f(x) \lesssim \lim_{r \to \infty}\rho(r)Mf(x)
\lesssim
\|f\|_{L^{\Phi}}\Psi^{-1}\circ\Phi\left(\frac{Mf(x)}{C_0\|f\|_{L^{\Phi}}}\right).
\]
It thus remains to handle the case where
\[
0<\Phi\left(\frac{Mf(x)}{C_0\|f\|_{L^{\Phi}}}\right)<\infty.
\]
In the case we can
choose $r>0$ such that
\[
r^{-n}=\Phi\left(\frac{Mf(x)}{C_0\|f\|_{L^{\Phi}}}\right).
\]
Let $B=B(x,r)$ and represent $f$ as
\begin{equation*}
f=f_1+f_2, \ \quad f_1=f\chi_{B},\quad
 f_2=f\chi_{{\mathbb R}^n \setminus B}
\end{equation*}
so that
$M_\rho f(x)\le M_\rho f_1(x)+M_\rho f_2(x)$.

We have (\ref{eq:180314-2})
for $f_1$.
Meanwhile by Lemma~\ref{lemHold}, 
\begin{equation*}
\begin{split}
M_{\rho} f_2(x) & = \sup_{t>0}\frac{\rho(t)}{|B(x,t)|} \int_{B(x,t)\cap {{\mathbb R}^n \setminus B(x,r)}}|f(z)|d z
\\
& =\, \sup_{r<t<\infty}\frac{\rho(t)}{|B(x,t)|} \int_{B(x,t)}|f(z)|d z
\\
&\lesssim \sup_{r<t<\infty} \rho(t) \, \Phi^{-1}(|B(x,t)|^{-1}) \, \|f\|_{L^{\Phi}(B(x,t))} \\
&\lesssim \|f\|_{L^{\Phi}} \sup_{r<t<\infty} \rho(t) \, \Phi^{-1}(t^{-n}).
\end{split}
\end{equation*}
Consequently we have by Lemma \ref{swkshr}
\begin{equation*}
M_\rho f(x) \lesssim
\rho(r) Mf(x)+\|f\|_{L^{\Phi}} \sup_{r<t<\infty} \rho(t) \, \Phi^{-1}(t^{-n}).
\end{equation*}
Thus, by \eqref{adRieszCharOrl2GFM}
and the monotonicity of $\Psi^{-1}$
we obtain
\begin{align*}
M_{\rho} f(x)  \lesssim  Mf(x)\frac{\Psi^{-1}(r^{-n})}{\Phi^{-1}(r^{-n})} + \|f\|_{L^{\Phi}} \, \Psi^{-1}(r^{-n}).
\end{align*}
Since $\Phi^{-1}(r^{-n})=\dfrac{Mf(x)}{C_0\|f\|_{L^{\Phi}}}$,
we have
\begin{align*}
 M_{\rho} f(x)
 \lesssim \|f\|_{L^{\Phi}} \, \Psi^{-1}(r^{-n})
 = \|f\|_{L^{\Phi}} \, \Psi^{-1}\left(\Phi\left(\frac{Mf(x)}{C_0\|f\|_{L^{\Phi}}}\right)\right).
\end{align*}
Therefore, we get
(\ref{eq:180314-7}).
\end{proof}

In \cite{HNS}
we obtain a counterpart to generalized Orlicz--Morrey spaces
of the second kind defined in \cite{GST}.
However, as is written in \cite{GST}
generalized Orlicz--Morrey spaces
of the second kind do not cover $L^2({\mathbb R}^n) \cap L^3({\mathbb R}^n)$.
So, the following theorem can be viewed as a different theorem
from \cite{GST}:

\begin{theorem}\label{AdamsFrMaxCharOrlGFM}
Let $\Phi, \Psi$ be Young functions.
Assume that $\rho$ is increasing
and that $r\mapsto r^{-n}\rho(r)$ is decreasing.
Then the condition \eqref{adRieszCharOrl2GFM} is necessary and sufficient
for the boundedness of $M_{\rho}$ from $L^{\Phi}({\mathbb R}^n)$ to ${\rm W}L^{\Psi}({\mathbb R}^n)$.
Moreover, if $\Phi\in\nabla_2,$ then
the condition \eqref{adRieszCharOrl2GFM} is necessary and sufficient
for the boundedness of $M_{\rho}$ from $L^{\Phi}({\mathbb R}^n)$ to $L^{\Psi}({\mathbb R}^n)$.
\end{theorem}

\begin{proof}
We start with the necessity.
For the necessity,
we can concentrate on the boundedness
 of $M_{\rho}$ from $L^{\Phi}({\mathbb R}^n)$ to ${\rm W}L^{\Psi}({\mathbb R}^n)$,
since the  boundedness
 of $M_{\rho}$ from $L^{\Phi}({\mathbb R}^n)$ to $L^{\Psi}({\mathbb R}^n)$
is stronger than the boundedness
 of $M_{\rho}$ from $L^{\Phi}({\mathbb R}^n)$ to ${\rm W}L^{\Psi}({\mathbb R}^n)$.
With this in mind,
assume that $M_{\rho}$ is bounded from $L^{\Phi}({\mathbb R}^n)$ to ${\rm W}L^{\Psi}({\mathbb R}^n)$. We utilize a trivial pointwise estimate
\begin{equation}\label{trpwesfr}
\rho(r)\chi_{B(0,r)}\leq  M_{\rho} \chi_{B(0,2r)}.
\end{equation}
Therefore, by the doubling property of $\Phi^{-1}$ and Lemma \ref{charorlc}, we have
\begin{align*}
\rho(r)&\lesssim \Psi^{-1}(r^{-n})\|M_{\rho} \chi_{B(0,2r)}\|_{{\rm W}L^{\Psi}(B(0,r))}\\
&\lesssim \Psi^{-1}(r^{-n})\|M_{\rho} \chi_{B(0,2r)}\|_{{\rm W}L^{\Psi}}
\\
&\lesssim \Psi^{-1}(r^{-n})\|\chi_{B(0,2r)}\|_{L^{\Phi}}\\
&\lesssim \frac{\Psi^{-1}(r^{-n})}{\Phi^{-1}(r^{-n})}.
\end{align*}

We move on to the sufficiency.
Here and below we let $f$ be a nonzero measurable function.
\begin{itemize}
\item
Let $C_0$ be as in \eqref{MbdninqW}. Then by \eqref{MbdninqW} and \eqref{eq:180314-7}, we have
\begin{align*}
&\sup_{r>0}\Psi(r)\, m\Big(\frac{M_{\rho} f(y)}{C_1\|f\|_{L^{\Phi}}},r\Big)
=\sup_{r>0}r\, m\Big(\Psi\Big(\frac{M_{\rho} f(y)}{C_1\|f\|_{L^{\Phi}}}\Big),r\Big)
\\
\leq &\sup_{r>0}r\, m\Big(\Phi\Big(\frac{M f(y)}{C_0\|f\|_{L^{\Phi}}}\Big),r\Big)
\leq\sup_{r>0}\Phi(r)\, m\Big(\frac{M f(y)}{\|Mf\|_{{\rm W}L^{\Phi}}},r\Big)\leq 1,
\end{align*}
i.e.
\begin{equation}\label{gfhajyufrGFM}
\|M_{\rho}f\|_{{\rm W}L^{\Psi}}\lesssim \|f\|_{L^{\Phi}}.
\end{equation}

\item
Assume in addition that $\Phi\in\nabla_2$.
Let $C_0$ be as in \eqref{MbdninqS}. By \eqref{MbdninqS} and \eqref{eq:180314-7}, we have
\begin{align*}
\int_{{\mathbb R}^n}\Psi\left(\frac{M_{\rho} f(y)}{C_1\|f\|_{L^{\Phi}}}\right)dy
&\leq  \int_{{\mathbb R}^n}\Phi\left(\frac{M f(y)}{C_0\|f\|_{L^{\Phi}}}\right)dy \\
&\leq  \int_{{\mathbb R}^n}\Phi\left(\frac{M f(y)}{\|Mf\|_{L^{\Phi}}}\right)dy \leq 1,
\end{align*}
i.e.
\begin{equation}\label{gfhajGFM}
\|M_{\rho}f\|_{L^{\Psi}}\lesssim \|f\|_{L^{\Phi}}.
\end{equation}
\end{itemize}
\end{proof}

\section{Generalized Orlicz--Morrey spaces\\ of the third kind}\label{sec:Orl-Mor}

In \cite{DerGulSam}, the generalized Orlicz--Morrey space $M^{\Phi,\varphi}({\mathbb R}^n)$ was introduced
to unify Orlicz spaces and generalized Morrey spaces.
Other definitions of generalized Orlicz--Morrey spaces can be found in \cite{Nakai2004KIT,SawSugTan}.
In words of \cite{GulHasSawNak},
our generalized Orlicz--Morrey space is the third kind and the ones in \cite{Nakai2004KIT}
and \cite{SawSugTan} are the first kind and the second kind, respectively.
Notice that the definition of the space of the third kind relies only on the fact that $L^{\Phi}({\mathbb R}^n)$
is a normed linear space, which is independent of the condition that it is generated by modulars. 

The definition of generalized Orlicz--Morrey spaces of the third kind is as follows:
\begin{definition}
Let $\varphi$ be a positive measurable function on $(0,\infty)$ and $\Phi$ any Young function.
We denote by $\mathcal{M}^{\Phi,\varphi}({\mathbb R}^n)$ the generalized Orlicz--Morrey space
 of the third kind, the space of all
functions $f\in L^{\Phi}_{\rm loc}({\mathbb R}^n)$ with finite norm
$$
\|f\|_{\mathcal{M}^{\Phi,\varphi}} = \sup\limits_{x\in{\mathbb R}^n, r>0}
\varphi(r)^{-1} \Phi^{-1}(r^{-n}) \|f\|_{L^{\Phi}(B(x,r))}.
$$
Also by ${\rm W}\mathcal{M}^{\Phi,\varphi}({\mathbb R}^n)$ we denote the weak generalized Orlicz--Morrey space of the third kind of all measurable functions $f\in {\rm W}L^{\Phi}_{\rm loc}({\mathbb R}^n)$ for which
$$
\|f\|_{{\rm W}\mathcal{M}^{\Phi,\varphi}}
=
\sup\limits_{x\in{\mathbb R}^n, r>0} \varphi(r)^{-1} \Phi^{-1}(r^{-n}) \|f\|_{{\rm W}L^{\Phi}(B(x,r))} < \infty.
$$
\end{definition}

A function $\varphi:(0,\infty) \to (0,\infty)$ is said to be almost increasing (resp.
almost decreasing) if there exists a constant $C > 0$ such that
$$
\varphi(r)\leq C \varphi(s)\qquad (\text{resp. }\varphi(r)\geq C \varphi(s))\quad \text{for  } r\leq s.
$$
For a Young function $\Phi$, we denote by ${\mathcal{G}}_{\Phi}$
the set of all $\varphi:(0,\infty) \to (0,\infty)$ functions
such that $t\in(0,\infty) \mapsto \frac{\varphi(t)}{\Phi^{-1}(t^{-n})}$
is almost increasing and $t\in(0,\infty) \mapsto \frac{\varphi(t)}{\Phi^{-1}(t^{-n})t^n}$ is almost decreasing.
Note that $\varphi\in{\mathcal{G}}_{\Phi}$ implies doubling condition of $\varphi$.

We investigate the structure of $ \mathcal{M}^{\Phi,\varphi}({\mathbb R}^n)$.
We denote by $\Theta$ the set of all measurable functions equivalent to $0$ on ${\mathbb R}^n$.
To exclude some trivial cases,
we can use
the following lemma was proved in \cite{DerGulHasFr}:
\begin{lemma}\label{Lemma1Orl}
Let $\Phi$ be a Young function and $ \varphi $ be a positive measurable function on $(0,\infty)$.
\begin{enumerate}
\item[\rm(i)] If
\begin{align}\label{L11Orl}
\sup_{ t<r<\infty }\frac{\Phi^{-1}(r^{-n})}{\varphi(r)}=\infty\quad\textrm{ for some } t>0
\end{align}
then $ \mathcal{M}^{\Phi,\varphi}({\mathbb R}^n)=\Theta$.
\item[\rm(ii)] If
\begin{align}\label{L12Orl}
\sup_{ 0<r<\tau} \varphi(r)^{-1}=\infty\quad\textrm{ for some } \tau>0
\end{align}
then $ \mathcal{M}^{\Phi,\varphi}({\mathbb R}^n)=\Theta $.
\end{enumerate}
\end{lemma}

\begin{remark}\label{ntvgom}
If
\[
\sup_{0<r \le t}
\frac{\Phi^{-1}(r^{-n})r^n}{\varphi(r)}
=\infty\quad\text{for some } t>0,
\]
then
${\mathcal M}^{\Phi,\varphi}=\Theta$.
Actually, by Remark~\ref{dfghj} one has
$$
\Phi^{-1}(r^{-n})r^n\le\Phi^{-1}(t^{-n})t^n<\infty
$$ and
then
$$
 \sup_{0<r<t} \varphi(r)^{-1}=\infty.
$$
\end{remark}

\begin{remark}
Based on Lemma~\ref{Lemma1Orl} and Remark~\ref{ntvgom}
and an observation similar to the one made by Nakai \cite[p. 446]{Nakai2000SCM}
it can be assumed that $\varphi \in\mathcal{G}_{\Phi}$
in the definition of $\mathcal{M}^{\Phi,\varphi}({\mathbb R}^n)$.
More explicitly,
we have the following observation:
%

\noindent{\rm (i)}
By Lemma~\ref{Lemma1Orl}
we may assume that $\displaystyle\inf_{r\le t<\infty}\frac{\varphi(t)}{\Phi^{-1}(t^{-n})}>0$ for every $r>0$.
Let
\begin{equation*}
 \psi(r)=\Phi^{-1}(r^{-n})\displaystyle\inf_{r\le t<\infty}\frac{\varphi(t)}{\Phi^{-1}(t^{-n})},
 \quad r>0.
\end{equation*}
Then $r\in(0,\infty) \mapsto \frac{\psi(r)}{\Phi^{-1}(r^{-n})}$ is increasing
and
$\mathcal{M}^{\Phi,\varphi}({\mathbb R}^n)=\mathcal{M}^{\Phi,\psi}({\mathbb R}^n)$ with equivalent norms.
Indeed, it is clear that $\psi(r)\le\varphi(r)$ by the definition of $\psi$.
Hence
${\mathcal M}^{\Phi,\psi}({\mathbb R}^n)\subset{\mathcal M}^{\Phi,\varphi}({\mathbb R}^n)$
and
$\|f\|_{\mathcal{M}^{\Phi,\varphi}}\le\|f\|_{\mathcal{M}^{\Phi,\psi}}$.
On the other hand,
\begin{align*}
 \sup_{r>0}\psi(r)^{-1}\Phi^{-1}(r^{-n})\|f\|_{L^{\Phi}(B(x,r))}
 &=\sup_{r>0}\frac1{\inf_{r\le t<\infty}\frac{\varphi(t)}{\Phi^{-1}(t^{-n})}} \|f\|_{L^{\Phi}(B(x,r))} \\
 &=\sup_{r>0}\left(\sup_{r\le t<\infty}\frac{\Phi^{-1}(t^{-n})}{\varphi(t)}\right) \|f\|_{L^{\Phi}(B(x,r))}\\
 &\le\sup_{t>0}\varphi(t)^{-1} \Phi^{-1}(t^{-n})\|f\|_{L^{\Phi}(B(x,t))}.
\end{align*}
Hence
${\mathcal M}^{\Phi,\varphi}({\mathbb R}^n)\subset{\mathcal M}^{\Phi,\psi}({\mathbb R}^n)$
and
$\|f\|_{\mathcal{M}^{\Phi,\psi}}\le\|f\|_{\mathcal{M}^{\Phi,\varphi}}$.

\noindent{\rm (ii)}
By Remark~\ref{ntvgom}
we may assume that $\displaystyle\inf_{0<t\le r}\frac{\varphi(t)}{\Phi^{-1}(t^{-n})t^n}>0$ for every $r>0$.
Define $\psi(r)$ by the formula:
\[
\sup_{t \in (0,r]}
\frac{\Phi^{-1}(t^{-n})t^n}{\varphi(t)}
=
\frac{\Phi^{-1}(r^{-n})r^n}{\psi(r)}.
\]
It is easy to see that
$\psi(r) \le \varphi(r)$
for any $r>0$.
Thus,
${\mathcal M}^{\Phi,\psi}({\mathbb R}^n)\subset{\mathcal M}^{\Phi,\varphi}({\mathbb R}^n)$
and
$\|f\|_{\mathcal{M}^{\Phi,\varphi}}\le\|f\|_{\mathcal{M}^{\Phi,\psi}}$.
Conversely,
let $f \in {\mathcal M}^{\Phi,\varphi}({\mathbb R}^n)$.
For any $r\in(0,\infty)$,
choose $t \in (0,r]$ so that
\[
\frac{\Phi^{-1}(r^{-n})r^n}{\psi(r)}
\le 2
\frac{\Phi^{-1}(t^{-n})t^n}{\varphi(t)},
\]
and cover $B(x,r)$ with a family of $N$ balls
$\{B(x_j,t)\}_{j=1}^N$,
where $N \lesssim r^{-n}t^n$.
Let $j_0$ be such that
\begin{align*}
\|f\|_{L^\Phi(B(x,r))}
&\le N\|f\|_{L^\Phi(B(x_{j_0},t))}\\
&=N\max_{j=1,2\ldots,N}\|f\|_{L^\Phi(B(x_{j},t))}\\
&\lesssim t^{-n}r^n\|f\|_{L^\Phi(B(x_{j_0},r))}.
\end{align*}
Thus,
\begin{align*}
\frac{\Phi^{-1}(r^{-n})}{\psi(r)}
\|f\|_{L^\Phi(B(x,r))}
&\lesssim
\frac{r^n\Phi^{-1}(r^{-n})}{\psi(r)t^n}
\|f\|_{L^\Phi(B(x_{j_0},t))}\\
&\le
2\frac{\Phi^{-1}(t^{-n})}{\varphi(t)}
\|f\|_{L^\Phi(B(x_{j_0},t))}\\
&\le 2
\|f\|_{{\mathcal M}^{\Phi,\varphi}},
\end{align*}
implying
$f \in {\mathcal M}^{\Phi,\varphi}({\mathbb R}^n)$
and
$\|f\|_{\mathcal{M}^{\Phi,\psi}}\le\|f\|_{\mathcal{M}^{\Phi,\varphi}}$.
Thus,
$
{\mathcal M}^{\Phi,\psi}({\mathbb R}^n)
\hookrightarrow
{\mathcal M}^{\Phi,\varphi}({\mathbb R}^n).
$
\end{remark}
As the following lemma shows,
${\mathcal{G}}_{\Phi}$
is useful:
\begin{lemma}[\cite{GulDerPot}]\label{charOrlMor}
Let $B_0:=B(x_0,r_0)$. If $\varphi\in{\mathcal{G}}_{\Phi}$ is almost decreasing, then there exist $C>0$ such that
$$
\frac{1}{\varphi(r_0)}\leq \|\chi_{B_0}\|_{\mathcal{M}^{\Phi,\varphi}}\leq \frac{C}{\varphi(r_0)}.
$$
\end{lemma}

\section{Generalized fractional integrals on generalized Orlicz--Morrey spaces}\label{sec:Ir Orl-Mor}

We remark that there are two types of the boundedness
of the fractional integral operators.
One is the Spanne-type boundedness obtained in
\cite{Spanne65}.
Another boundedness is of Adams-type obtained by Adams
\cite{Adams75-2}.
In the classical case due to the fact that Morrey spaces
are nested
we can say that the Adams-type boundedness is stronger
than the Spanne-type boundedness.
However, we need to depend on the pointwise estimate
of Hedberg-type
\cite{Hedberg},
so the Adams-type boundedness is unavailable
for local Morrey spaces.
In this section we give a characterization
for the Spanne-type boundedness
and
the Adams-type boundedness of the operator $I_{\rho}$
on generalized Orlicz--Morrey spaces, respectively.

\subsection{Spanne-type result}

We need the  following lemma is valid:
\begin{lemma}\label{lem3.3.Pot2}
Let $\Phi, \Psi$ be Young functions,
and let $\rho$ satisfy \eqref{int rho} and \eqref{sup rho}.
Assume that the condition \eqref{adRieszCharOrl1} is fulfilled.
Then there exists a positive constant $C$ such that,
for all $f\in L^{\Phi}_{\rm loc}({\mathbb R}^n)$ and $B=B(x,r)$,
\begin{align}\label{eq3.5.w}
\lefteqn{
 \|I_{\rho} f\|_{{\rm W}L^{\Psi}(B)}
}\nonumber\\
& \le C\|f\|_{L^{\Phi}(2B)} +
 \frac{C}{\Psi^{-1}\big(r^{-n}\big)}
 \int_{2k_1r}^{\infty} \|f\|_{L^{\Phi}(B(x,t))} \, \Phi^{-1}\big(t^{-n}\big)\rho(t) \frac{dt}{t}.
\end{align}
Moreover if we assume $\Phi\in\nabla_2$, the following inequality is also valid:
\begin{align}\label{eq3.5.}
\lefteqn{
 \|I_{\rho} f\|_{L^{\Psi}(B)}
}\nonumber\\
& \le C\|f\|_{L^{\Phi}(2B)} +
 \frac{C}{\Psi^{-1}\big(r^{-n}\big)}
 \int_{2k_1r}^{\infty} \|f\|_{L^{\Phi}(B(x,t))} \, \Phi^{-1}\big(t^{-n}\big)\rho(t) \frac{dt}{t}.
\end{align}
\end{lemma}

\begin{proof}
We represent  $f$ as
\begin{equation*}
f=f_1+f_2, \ \quad f_1=f\chi _{2B},\quad
 f_2=f-f_1, \ \quad r>0.
\end{equation*}
Then we have
$$
\|I_\rho f\|_{WL^{\Psi}(B)} \le 2(\|I_\rho f_1\|_{WL^{\Psi}(B)} +\|I_\rho f_2\|_{WL^{\Psi}(B)}).
$$

From the
boundedness of $I_\rho$ from $L^{\Phi}({\mathbb R}^n)$ to $WL^{\Psi}({\mathbb R}^n)$ (see Theorem \ref{AdGulGenRszOrlMorNec})
it follows that:
\begin{equation}\label{ghjklyu}
 \|I_\rho f_1\|_{WL^{\Psi}(B)}
 \lesssim \|I_\rho f_1\|_{WL^{\Psi}({\mathbb R}^n)}\leq
 C\|f_1\|_{L^{\Phi}({\mathbb R}^n)}=C\|f\|_{L^{\Phi}(2B)},
\end{equation}
where constant $C >0$ is independent of $f$.

For $f_2$ we have
\begin{equation*}
 \left|I_{\rho} f_2(y)\right|
 \le
 \int_{{\,^{^{\complement}}\!}(2B)} \frac{\rho(|y-z|)}{|y-z|^{n}} |f(z)| \,dz
 =
 \sum_{j=1}^{\infty}
 \int_{2^{j+1}B\setminus 2^{j}B} \frac{\rho(|y-z|)}{|y-z|^{n}} |f(z)| \,dz.
\end{equation*}
A geometric observation shows that $y\in B$, $z\in {\mathbb R}^n \setminus (2B)$ implies
$$
\frac{1}{2}|x-z|\le |y-z|\le\frac{3}{2}|x-z|.
$$
Using \eqref{sup rho} and Lemma \ref{lemHold},
we have
\begin{align*}
 &\int_{2^{j+1}B\setminus 2^{j}B} \frac{\rho(|y-z|)}{|y-z|^{n}} |f(z)| \,dz \\
 &\lesssim
 \left(\sup_{2^{j-1}r\le t\le 3\cdot2^jr}\rho(t)\right)
 \frac1{|2^jB|}\int_{2^{j+1}B\setminus 2^{j}B} |f(z)| \,dz  \\
 &\lesssim
 \int_{2^j k_1 r}^{3\cdot2^j k_2 r}\frac{\rho(t)}{t}\,dt \
 \|f\|_{L^{\Phi}(B(x,2^{j+1}r))} \, \Phi^{-1}\big(|2^{j+1}B|^{-1}\big) \\
 &\lesssim
 \int_{2^j k_1 r}^{2^j\max(3k_2,4) r}\|f\|_{L^{\Phi}(B(x,t))}\, \Phi^{-1}\big(t^{-n}\big)\frac{\rho(t)}{t}\,dt.
\end{align*}
Then
\begin{equation}\label{sal00Pot2}
 \left|I_{\rho} f_2(y)\right|
 \lesssim
 \int_{2k_1r}^{\infty}\|f\|_{L^{\Phi}(B(x,t))}\, \Phi^{-1}\big(t^{-n}\big)\frac{\rho(t)}{t}\,dt.
\end{equation}
Thus by Lemma \ref{charorlc} we have
\begin{equation}\label{asdertg}
\|I_{\rho} f_2\|_{WL^{\Psi}(B)} \lesssim  \frac{1}{\Psi^{-1}\big(r^{-n}\big)} \int_{2k_1r}^{\infty} \|f\|_{L^{\Phi}(B(x,t))} \, \Phi^{-1}\big(t^{-n}\big)\rho(t) \frac{dt}{t}.
\end{equation}
Therefore we obtain \eqref{eq3.5.w} by \eqref{ghjklyu} and \eqref{asdertg}.

If $\Phi\in\nabla_2$, then we can use strong type inequality instead of \eqref{ghjklyu}
and obtain \eqref{eq3.5.} by using the same argument.
\end{proof}

\begin{remark}
In the case $\Phi(t)=t^p~(1\le p<\infty)$ Lemma \ref{lem3.3.Pot2} was proved in \cite{GIKS2015}.
\end{remark}

The following theorem gives necessary and sufficient conditions for Spanne-type
boundedness of the operator $I_{\rho}$ from $\mathcal{M}^{\Phi,\varphi_1}({\mathbb R}^n)$ to $\mathcal{M}^{\Psi,\varphi_2}({\mathbb R}^n)$.

\begin{theorem}[Spanne-type result]\label{SpGulGenRszOrlMorNec}
Let $\Phi, \Psi$ be Young functions,
and let $\varphi_1\in{\mathcal{G}}_{\Phi}$ and $\varphi_2\in{\mathcal{G}}_{\Psi}$.
\begin{enumerate}
\item
Let $\rho$ satisfy \eqref{int rho} and \eqref{sup rho}.
Assume that \eqref{adRieszCharOrl1} is fulfilled.
Then the conditions
\begin{equation}\label{VZACSnox}
\frac{\varphi_1(r)}{\Phi^{-1}\big(r^{-n}\big)}   \le  C \,\frac{\varphi_2\big(r\big)}{\Psi^{-1}\big(r^{-n}\big)},
\end{equation}
\begin{equation}\label{eq3.6.VZPotnox}
\int_{r}^{\infty} \varphi_1(t) \rho(t) \frac{dt}{t} \le  C \,\varphi_2(r),
\end{equation}
for all $r>0$, where $C>0$ does not depend on $r$, are sufficient for the boundedness of $I_{\rho}$ from $\mathcal{M}^{\Phi,\varphi_1}({\mathbb R}^n)$ to ${\rm W}\mathcal{M}^{1,\varphi_2}({\mathbb R}^n)$.
Moreover, if $\Phi\in\nabla_2$,
then the conditions
\eqref{VZACSnox} and \eqref{eq3.6.VZPotnox}
are sufficient for the boundedness of $I_{\rho}$ from $\mathcal{M}^{\Phi,\varphi_1}({\mathbb R}^n)$ to $\mathcal{M}^{\Psi,\varphi_2}({\mathbb R}^n)$.

\item
Let $\varphi_1$ be almost decreasing.
Then the condition
\begin{equation}\label{condSpan}
\varphi_1(r) \int_{0}^{r}\frac{\rho(t)}{t}dt\le C \varphi_2(r),
\end{equation}
for all $r>0$, where $C>0$ does not depend on $r$,
is necessary for the boundedness of $I_{\rho}$
from $\mathcal{M}^{\Phi,\varphi_1}({\mathbb R}^n)$ to ${\rm W}\mathcal{M}^{1,\varphi_2}({\mathbb R}^n)$
and hence $\mathcal{M}^{\Phi,\varphi_1}({\mathbb R}^n)$ to $\mathcal{M}^{\Psi,\varphi_2}({\mathbb R}^n)$.

\item
Let $\rho$ satisfy \eqref{int rho} and \eqref{sup rho}.
Assume that \eqref{adRieszCharOrl1} is fulfilled,
that $\varphi_1$ is almost decreasing
and that $\varphi_1$ and $\varphi_2$ satisfy \eqref{VZACSnox}.
Assume also that $\varphi_1$ and $\rho$ satisfy the condition
\begin{equation}\label{intcondSpannox}
\int_{r}^{\infty} \varphi_1(t) \rho(t) \frac{dt}{t} \le  C \,\varphi_1(r) \rho(r),
\end{equation}
for all $r>0$, where $C>0$ does not depend on $r$.
Then the condition \eqref{condSpan}
is necessary and sufficient for the boundedness of $I_{\rho}$
from $\mathcal{M}^{\Phi,\varphi_1}({\mathbb R}^n)$ to ${\rm W}\mathcal{M}^{\Psi,\varphi_2}({\mathbb R}^n)$.
Moreover, if $\Phi\in\nabla_2$,
then the condition \eqref{condSpan}
is necessary and sufficient for the boundedness of $I_{\rho}$
from $\mathcal{M}^{\Phi,\varphi_1}({\mathbb R}^n)$ to $\mathcal{M}^{\Psi,\varphi_2}({\mathbb R}^n)$.
\end{enumerate}
\end{theorem}

\begin{proof}
{\it 1.} By \eqref{eq3.5.w}, \eqref{VZACSnox} and \eqref{eq3.6.VZPotnox} we have
\begin{align*}
 \|I_{\rho} f\|_{{\rm W}\mathcal{M}^{\Psi,\varphi_2}}
 & \lesssim
 \sup_{x\in{\mathbb R}^n, r>0}\varphi_2(r)^{-1} \Psi^{-1}\big(r^{-n}\big) \|f\|_{L^{\Phi}(B(x,2r))} \\
 &\phantom{**}+
 \sup_{x\in{\mathbb R}^n, r>0} \varphi_2(r)^{-1}
  \int_{2k_1r}^{\infty} \|f\|_{L^{\Phi}(B(x,t))} \, \Phi^{-1}\big(t^{-n}\big)\rho(t) \frac{dt}{t} \\
 & \lesssim \sup_{x\in{\mathbb R}^n, r>0} \varphi_1(r)^{-1} \Phi^{-1}\big(r^{-n}\big) \|f\|_{L^{\Phi}(B(x,r))} \\
 &\phantom{**}+
 \sup_{x\in{\mathbb R}^n, r>0} \varphi_2(r)^{-1}
  \int_{2k_1r}^{\infty} \varphi_1(t)^{-1}\rho(t) \frac{dt}{t} \ \|f\|_{\mathcal{M}^{\Phi,\varphi_1}}\\
 & \lesssim
 \|f\|_{\mathcal{M}^{\Phi,\varphi_1}}.
\end{align*}
Simply replace
${\rm W}L^\Psi(B)$ with $L^\Psi(B)$
and
${\rm W}\mathcal{M}^{\Psi,\eta}({\mathbb R}^n)$ with $\mathcal{M}^{\Psi,\eta}({\mathbb R}^n)$ and use \eqref{eq3.5.}, \eqref{VZACSnox} and \eqref{eq3.6.VZPotnox} for the strong estimate.

\medskip

{\it 2.}
We will now prove the second part. Let $B_R=B(0,R)$ and $x\in B_{R/2}$. By Lemma \ref{pwsgenfr} we have
$$\rho^{*}(R/2):=\int_{0}^{R/2}\frac{\rho(t)}{t}dt\leq C I_{\rho} \chi_{B_R}(x).$$
Therefore, by Lemma \ref{charOrlMor} and the doubling property of $\varphi_1$,
\begin{align*}
\rho^{*}(R/2)&\lesssim|B_{R/2}|^{-1}\|I_{\rho} \chi_{B_R}\|_{WL^{1}(B_{R/2})}
\lesssim\varphi_2(R/2)\|I_{\rho} \chi_{B_R}\|_{{\rm W}\mathcal{M}^{1,\varphi_2}} \\
&\lesssim\varphi_2(R/2)\|\chi_{B_0}\|_{\mathcal{M}^{\Phi,\varphi_1}}\lesssim\frac{\varphi_2(R/2)}{\varphi_1(R)}
\lesssim \frac{\varphi_2(R/2)}{\varphi_1(R/2)}.
\end{align*}
Since this is true for every $R>0$, we are done.

\medskip

{\it 3.}
The third statement of the theorem follows from the first and second parts of the theorem.
\end{proof}

\subsection{Adams-type result}

The following theorem was proved in \cite[Theorem 4.6]{DerGulSam}:
\begin{theorem}\label{cor4.4.}
\
\begin{enumerate}
\item
Let $\varphi\in{\mathcal{G}}_{\Phi}$ be almost decreasing.
Then the maximal operator $M$ is bounded
from $\mathcal{M}^{\Phi,\varphi}({\mathbb R}^n)$ to ${\rm W}\mathcal{M}^{\Phi,\varphi}({\mathbb R}^n)$.
\item
Let $\Phi\in \nabla_2$ and $\varphi\in{\mathcal{G}}_{\Phi}$ be almost decreasing.
Then the maximal operator $M$ is bounded
on $\mathcal{M}^{\Phi,\varphi}({\mathbb R}^n)$.
\end{enumerate}
\end{theorem}
The following theorem gives necessary and sufficient conditions for Adams-type
boundedness of the operator $I_{\rho}$
from $\mathcal{M}^{\Phi,\varphi}({\mathbb R}^n)$ to $\mathcal{M}^{\Psi,\eta}({\mathbb R}^n)$:

\begin{theorem}[Adams-type result]\label{AdGulGenRszOrlMorNecG}
Let $\Phi$ be a Yougn function, and let $\varphi\in{\mathcal{G}}_{\Phi}$ be almost decreasing.
Let $\beta\in(0,1)$ and define $\eta(t)\equiv\varphi(t)^{\beta}$ for $t>0$
and $\Psi(t)\equiv\Phi(t^{1/\beta})$ for $t>0$.
\begin{enumerate}
\item
Let $\rho$ satisfy \eqref{int rho} and \eqref{sup rho}.
Then the condition
\begin{equation}\label{eq3.6.V}
\varphi(r)\int_{0}^{r}\frac{\rho(t)}{t}dt + \int_{r}^{\infty}\rho(t) \, \varphi(t) \frac{dt}{t} \le C
\eta(r),
\end{equation}
for all $r>0$, where $C>0$ does not depend on $r$, is sufficient for the boundedness of $I_{\rho}$ from $\mathcal{M}^{\Phi,\varphi}({\mathbb R}^n)$ to ${\rm W}\mathcal{M}^{\Psi,\eta}({\mathbb R}^n)$.
Moreover, if $\Phi\in\nabla_2$,
then the condition \eqref{eq3.6.V} is sufficient for the boundedness of $I_{\rho}$
from $\mathcal{M}^{\Phi,\varphi}({\mathbb R}^n)$ to $\mathcal{M}^{\Psi,\eta}({\mathbb R}^n)$.

\item
The condition
\begin{equation}\label{condAdams}
\varphi(r)\int_{0}^{r}\frac{\rho(t)}{t}dt\le C \eta(r),
\end{equation}
for all $r>0$, where $C>0$ does not depend on $r$, is necessary
for the boundedness of $I_{\rho}$ from $\mathcal{M}^{\Phi,\varphi}({\mathbb R}^n)$ 
to ${\rm W}\mathcal{M}^{1,\eta}({\mathbb R}^n)$
and hence
for the boundedness of $I_{\rho}$ from $\mathcal{M}^{\Phi,\varphi}({\mathbb R}^n)$ 
to $\mathcal{M}^{\Psi,\eta}({\mathbb R}^n)$.
\item
Let $\rho$ satisfy \eqref{int rho} and \eqref{sup rho}.
Assume that $\varphi$ satisfies the condition
\begin{equation}\label{intcondAdamsxc}
\int_{r}^{\infty}\rho(t) \, \varphi(t) \frac{dt}{t} \le C \rho(r)\varphi(r),
\end{equation}
for all $r>0$, where $C>0$ does not depend on $r$.
Then the condition \eqref{condAdams}
is necessary and sufficient for the boundedness of $I_{\rho}$
 from $\mathcal{M}^{\Phi,\varphi}({\mathbb R}^n)$ to ${\rm W}\mathcal{M}^{\Psi,\eta}({\mathbb R}^n)$.
Moreover, if $\Phi\in\nabla_2$,
then the condition \eqref{condAdams}
is necessary and sufficient for the boundedness of $I_{\rho}$
from $\mathcal{M}^{\Phi,\varphi}({\mathbb R}^n)$ to $\mathcal{M}^{\Psi,\eta}({\mathbb R}^n)$.
\end{enumerate}
\end{theorem}
We notice that
the function $\varphi$ and $\eta$
come into play, unlike Spanne-type.
Similar to Lemma \ref{lem:180314-5},
we have the following pointwise estimate:

\begin{lemma}
Let $\Phi$ be a Young function, $\varphi\in{\mathcal{G}}_{\Phi}$, $\beta\in(0,1)$,
$\eta(t)\equiv\varphi(t)^{\beta}$ and $\Psi(t)\equiv\Phi(t^{1/\beta})$.
If \eqref{eq3.6.V} holds,
then there exists a positive constant $C$ such that,
for all non-negative measurable functions $f$
and for every $x\in {\mathbb R}^n$,
\begin{equation}\label{poiwseest}
I_{\rho} f(x) \le C(Mf(x))^{\beta} \, \|f\|_{\mathcal{M}^{\Phi,\varphi}}^{1-\beta}.
\end{equation}
\end{lemma}

\begin{proof}
Let $\tilde{\rho}$ be defined by (\ref{eq:tilde rho}). We have
\begin{equation*}
I_{\rho} f(x)
\le C\left[\sum_{j=-\infty}^{-1}+\sum_{j=0}^{\infty}\frac{\tilde{\rho}(2^j r)}{(2^j r)^n}\int_{|x-y|<2^{j}r}f(y) dy\right]=C({\rm I}+{\rm II})
\end{equation*}
for given $x\in {\mathbb R}^n$ and $r>0$.
Thus from
(\ref{eq:180314-9}) and (\ref{eq:180314-10})
with $\tau=\varphi$
we deduce
\begin{align*}
{\rm I}&\le C \sum_{j=-\infty}^{-1}\tilde{\rho}(2^j r) Mf(x)\le C \left(\int_{0}^{k_2 r}\frac{\rho(s)}{s}ds\right) Mf(x)\\
{\rm II}&\le C \sum_{j=0}^{\infty}\tilde{\rho}(2^j r)\varphi\big(2^{j}r\big)\|f\|_{\mathcal{M}^{\Phi,\varphi}}\le C \|f\|_{\mathcal{M}^{\Phi,\varphi}} \int_{k_1 r}^{\infty} \varphi\big(s\big) \frac{\rho(s)}{s}ds.
\end{align*}
Consequently we have
\begin{equation*}
\begin{split}
I_\rho f(x) \lesssim \left(\int_{0}^{k_2 r}\frac{\rho(s)}{s}ds\right) Mf(x)+\|f\|_{\mathcal{M}^{\Phi,\varphi}} \int_{k_1 r}^{\infty} \varphi\big(s\big) \frac{\rho(s)}{s}ds.
\end{split}
\end{equation*}
Thus, the technique in \cite[p. 6492]{SawSugTan2} by \eqref{eq3.6.V} and the doubling property of $\varphi$ we obtain
\begin{align*}
I_{\rho} f(x) & \lesssim  \min \big\{ \varphi(r)^{\beta-1} Mf(x), \varphi(r)^{\beta} \|f\|_{\mathcal{M}^{\Phi,\varphi}}\big\}
\\
& \lesssim \sup\limits_{s>0} \min \big\{ s^{\beta-1} Mf(x), s^{\beta} \|f\|_{\mathcal{M}^{\Phi,\varphi}}\big\}
\\
& = (Mf(x))^{\beta} \, \|f\|_{\mathcal{M}^{\Phi,\varphi}}^{1-\beta},
\end{align*}
where we have used that the supremum is achieved when the minimum parts are balanced.
Hence we have
$I_{\rho} f(x) \lesssim (Mf(x))^{\beta} \, \|f\|_{\mathcal{M}^{\Phi,\varphi}}^{1-\beta}$.
\end{proof}

We have the following scaling law:
\begin{lemma}\label{lem:180314-6}
Let $\beta>0$. Let $\Psi$ and $\Phi$ be Yougn functions,
and let $B$ be a ball.
Then $\|\,|f|^{\beta}\|_{L^{\Psi}(B)}= \|f\|_{L^{\Phi}(B)}^{\beta}$
and $\|\,|f|^{\beta}\|_{{\rm W}L^{\Psi}(B)}= \|f\|_{{\rm W}L^{\Phi}(B)}^{\beta}$
for all measurable functions $f$.
\end{lemma}

\begin{proof}
Simply note that
$$
\int_B \Psi\left(\frac{|f(x)|^{\beta}}{\|f\|_{L^{\Phi}(B)}^{\beta}}\right)dx
=
\int_B \Phi\left(\frac{|f(x)|}{\|f\|_{L^{\Phi}(B)}}\right)dx
$$
for $L^{\Phi}(B)$.
The equality for weak spaces can be proved similarly.
\end{proof}

\begin{proof}[Proof of Theorem \ref{AdGulGenRszOrlMorNecG}]
{\it 1.}
\begin{itemize}
\item
We deal with the weak-type estimate. By using inequality \eqref{poiwseest} we have for an arbitrary ball $B$
$$
\|I_{\rho} f\|_{{\rm W}L^{\Psi}(B)}\lesssim \|(Mf)^{\beta}\|_{{\rm W}L^{\Psi}(B)}\, \|f\|_{\mathcal{M}^{\Phi,\varphi}}^{1-\beta}.
$$
Consequently by using this inequality and Lemma \ref{lem:180314-6}
we have
\begin{equation}\label{poiwseestnorm}
\|I_{\rho} f\|_{{\rm W}L^{\Psi}(B)}\lesssim \|M f\|_{{\rm W}L^{\Phi}(B)}^{\beta}\, \|f\|_{\mathcal{M}^{\Phi,\varphi}}^{1-\beta}.
\end{equation}
From Theorem \ref{cor4.4.} and \eqref{poiwseestnorm}, we get
\begin{align*}
\|I_{\rho}f\|_{{\rm W}\mathcal{M}^{\Psi,\eta}}&=\sup\limits_{B}\eta(r)^{-1}\Psi^{-1}(r^{-n})\|I_{\rho} f\|_{{\rm W}L^{\Psi}(B)}\\
&\lesssim  \|f\|_{\mathcal{M}^{\Phi,\varphi}}^{1-\beta}\, \sup\limits_{B}\eta(r)^{-1}\Psi^{-1}(r^{-n})\|Mf\|_{{\rm W}L^{\Phi}(B)}^{\beta}\\
&=\|f\|_{\mathcal{M}^{\Phi,\varphi}}^{1-\beta}\, \left(\sup\limits_{B}\varphi(r)^{-1}\Phi^{-1}(r^{-n})\|Mf\|_{{\rm W}L^{\Phi}(B)}\right)^{\beta} \\
&\lesssim \|f\|_{\mathcal{M}^{\Phi,\varphi}}.
\end{align*}
\item
Simply replace
${\rm W}L^\Psi(B)$ with $L^\Psi(B)$
and
${\rm W}\mathcal{M}^{\Psi,\eta}({\mathbb R}^n)$ with $\mathcal{M}^{\Psi,\eta}({\mathbb R}^n)$ for the strong estimate.
\end{itemize}

\medskip

{\it 2.}
We will now prove the second part.
Let $B_R=B(0,R)$ and $x\in B_{R/2}$.
By Lemmas \ref{charorlc}, \ref{pwsgenfr} and \ref{charOrlMor} and the doubling property of $\varphi$, we have
\begin{align*}
\rho^{*}(R/2)&\leq C|B_{R/2}|^{-1}\|I_{\rho} \chi_{B_R}\|_{{\rm W}L^{1}(B_{R/2})}
\leq C\eta(R/2)\|I_{\rho} \chi_{B_R}\|_{{\rm W}\mathcal{M}^{1,\eta}}
\\
&\leq C\eta(R/2)\|\chi_{B_R}\|_{\mathcal{M}^{\Phi,\varphi}}\leq C\frac{\eta(R/2)}{\varphi(R)}\leq C\frac{\eta(R/2)}{\varphi(R/2)}= C \varphi(R/2)^{\beta-1}.
\end{align*}
Since this is true for every $R>0$, the proof is complete.

\medskip

{\it 3.}
This part follows from the first and second parts.
\end{proof}

\section{Generalized fractional maximal operators on generalized Orlicz--Morrey spaces}\label{sec:Mr Orl-Mor}

In this section we give a characterization for the Spanne-type boundedness and the Adams-type
boundedness of the operator $M_{\rho}$ on generalized Orlicz--Morrey spaces, respectively.

\subsection{Spanne-type result}

We use the following lemma:
\begin{lemma}\label{lem3.3.Pot2Fr}
Let $\Phi, \Psi$ be Young functions.
Assume that $\rho$ is increasing
and that $r\mapsto r^{-n}\rho(r)$ is decreasing.
Assume also that the condition \eqref{adRieszCharOrl2GFM} is fulfilled.
Then there exists a positive constant $C$ such that,
for all $f\in L^{\Phi}_{\rm loc}({\mathbb R}^n)$ and $B=B(x,r)$,
\begin{align}\label{eq3.5.wFr}
\lefteqn{
\|M_{\rho} f\|_{{\rm W}L^{\Psi}(B)}
}\nonumber\\
&\le C\|f\|_{L^{\Phi}(2B)} +
\frac{C}{\Psi^{-1}\big(r^{-n}\big)}
\sup_{r<t<\infty} \|f\|_{L^{\Phi}(B(x,2t))} \, \Phi^{-1}\big(t^{-n}\big)\rho(t).
\end{align}
Moreover if we assume $\Phi\in\nabla_2$, the following inequality is also valid:
\begin{align}\label{eq3.5.Fr}
\lefteqn{
\|M_{\rho} f\|_{L^{\Psi}(B)}
}\nonumber\\
&\le C\|f\|_{L^{\Phi}(2B)} +
\frac{C}{\Psi^{-1}\big(r^{-n}\big)}
\sup_{r<t<\infty} \|f\|_{L^{\Phi}(B(x,2t))} \, \Phi^{-1}\big(t^{-n}\big)\rho(t).
\end{align}
\end{lemma}
\begin{proof}
We represent  $f$ as
\begin{equation*}
 f=f_1+f_2, \ \quad f_1=f\chi _{2B},\quad
 f_2=f-f_1, \ \quad r>0.
\end{equation*}
Then we have
$$
\|M_\rho f\|_{WL^{\Psi}(B)} \lesssim \|M_\rho f_1\|_{WL^{\Psi}(B)} +\|M_\rho f_2\|_{WL^{\Psi}(B)}.
$$

From the boundedness of $M_\rho$ from $L^{\Phi}({\mathbb R}^n)$ to $WL^{\Psi}({\mathbb R}^n)$
(see Theorem \ref{AdamsFrMaxCharOrlGFM}) it follows that:
\begin{equation}\label{ghjklyuFr}
\|M_\rho f_1\|_{WL^{\Psi}(B)}\leq \|M_\rho f_1\|_{WL^{\Psi}({\mathbb R}^n)}\leq
C\|f_1\|_{L^{\Phi}({\mathbb R}^n)}=C\|f\|_{L^{\Phi}(2B)},
\end{equation}
where constant $C >0$ is independent of $f$.

If $y\in B$ and $r<t$, then $B(y,t)\subset B(x,2t)$.
Then, using Lemma~\ref{lemHold}, we have
\begin{align} \label{sal00Pot2Fr}
M_{\rho} f_2(y) & = \sup_{t>0}\frac{\rho(t)}{|B(y,t)|} \int_{B(y,t)\setminus (B(x,2r))}|f(z)|d z \notag
\\
& \le \, \sup_{t>r}\frac{\rho(t)}{|B(x,t)|} \int_{B(x,2t)}|f(z)|d z \notag
\\
&\lesssim \sup_{r<t<\infty} \rho(t) \, \Phi^{-1}(t^{-n}) \, \|f\|_{L^{\Phi}(B(x,2t))}
\quad\text{for $y\in B$}.
\end{align}

Thus by Lemma \ref{charorlc} we have
\begin{equation}\label{asdertgFr}
\|M_{\rho} f_2\|_{WL^{\Psi}(B)} \lesssim  \frac{1}{\Psi^{-1}\big(r^{-n}\big)} \sup_{r<t<\infty} \rho(t) \, \Phi^{-1}(t^{-n}) \, \|f\|_{L^{\Phi}(B(x,2t))}.
\end{equation}
Therefore we obtain \eqref{eq3.5.wFr} by \eqref{ghjklyuFr} and \eqref{asdertgFr}.

If $\Phi\in\nabla_2$, then we can use strong type inequality instead of \eqref{ghjklyuFr}
and obtain \eqref{eq3.5.Fr} by using the same argument.
\end{proof}


The following theorem gives a necessary and sufficient condition for Spanne-type
boundedness of the operator $M_{\rho}$
from $\mathcal{M}^{\Phi,\varphi_1}({\mathbb R}^n)$ to $\mathcal{M}^{\Psi,\varphi_2}({\mathbb R}^n)$:
We notice that the requirement is the same as the Orlicz spaces.

\begin{theorem}[Spanne-type result]\label{SpGulGenRszOrlMorNecFr}
Let $\Phi, \Psi$ be Young functions,
and let $\varphi_1\in{\mathcal{G}}_{\Phi}$ and $\varphi_2\in{\mathcal{G}}_{\Psi}$.
\begin{enumerate}
\item
Assume that $\rho$ is increasing
and that $r\mapsto r^{-n}\rho(r)$ is decreasing.
Assume also that the conditions \eqref{adRieszCharOrl2GFM} and \eqref{VZACSnox} are satisfied.
Then the condition
\begin{equation}\label{eq3.6.VZPotnoxFr}
\sup_{r<t<\infty}\varphi_1(t) \rho(t) \le  C \,\varphi_2(r),
\end{equation}
for all $r>0$, where $C>0$ does not depend on $r$, are sufficient for the boundedness of $M_{\rho}$ from $\mathcal{M}^{\Phi,\varphi_1}({\mathbb R}^n)$ to ${\rm W}\mathcal{M}^{\Psi,\varphi_2}({\mathbb R}^n)$.
Moreover, if $\Phi\in\nabla_2$,
then the condition \eqref{eq3.6.VZPotnoxFr} is sufficient for the boundedness of $M_{\rho}$
from $\mathcal{M}^{\Phi,\varphi_1}({\mathbb R}^n)$ to $\mathcal{M}^{\Psi,\varphi_2}({\mathbb R}^n)$.

\item
Let $\varphi_1$ be almost decreasing.
Then the condition
\begin{equation}\label{condSpanFr}
\varphi_1(r) \rho(r)\le C \varphi_2(r),
\end{equation}
for all $r>0$, where $C>0$ does not depend on $r$, is necessary for the boundedness of $M_{\rho}$ from $\mathcal{M}^{\Phi,\varphi_1}({\mathbb R}^n)$ to ${\rm W}\mathcal{M}^{1,\varphi_2}({\mathbb R}^n)$ and hence $\mathcal{M}^{\Phi,\varphi_1}({\mathbb R}^n)$ to $\mathcal{M}^{\Psi,\varphi_2}({\mathbb R}^n)$.

\item
Assume that $\rho$ is increasing
and that $r\mapsto r^{-n}\rho(r)$ is decreasing.
Assume also that the conditions \eqref{adRieszCharOrl2GFM} and \eqref{VZACSnox} are satisfied.
Let $\varphi_1$ and $\varphi_2$ be almost decreasing.
Then the condition \eqref{condSpanFr}
is necessary and sufficient for the boundedness of $M_{\rho}$
from $\mathcal{M}^{\Phi,\varphi_1}({\mathbb R}^n)$ to ${\rm W}\mathcal{M}^{\Psi,\varphi_2}({\mathbb R}^n)$.
Moreover, if $\Phi\in\nabla_2$,
then the condition \eqref{condSpanFr}
is necessary and sufficient for the boundedness of $M_{\rho}$
from $\mathcal{M}^{\Phi,\varphi_1}({\mathbb R}^n)$ to $\mathcal{M}^{\Psi,\varphi_2}({\mathbb R}^n)$.
\end{enumerate}
\end{theorem}

\begin{proof}
{\it 1.}
By \eqref{VZACSnox}, \eqref{eq3.5.wFr}, \eqref{eq3.6.VZPotnoxFr}
and the doubling properties of $\varphi_1$ and $\Phi^{-1}$
we have
\begin{align*}
 \|M_{\rho} f\|_{{\rm W}\mathcal{M}^{\Psi,\varphi_2}}
 &\lesssim
 \sup_{x\in{\mathbb R}^n, r>0}\varphi_2(r)^{-1} \Psi^{-1}\big(r^{-n}\big) \|f\|_{L^{\Phi}(B(x,2r))} \\
 &\phantom{**}+
 \sup_{x\in{\mathbb R}^n, r>0} \varphi_2(r)^{-1}
  \sup_{r<t<\infty} \|f\|_{L^{\Phi}(B(x,2t))} \, \Phi^{-1}\big(t^{-n}\big)\rho(t) \\
 &\lesssim
 \sup_{x\in{\mathbb R}^n, r>0} \varphi_1(r)^{-1} \Phi^{-1}\big(r^{-n}\big) \|f\|_{L^{\Phi}(B(x,r))} \\
 &\phantom{**}+
 \sup_{x\in{\mathbb R}^n, r>0} \varphi_2(r)^{-1}
  \sup_{r<t<\infty} \varphi_1(t)\rho(t) \|f\|_{\mathcal{M}^{\Phi,\varphi_1}} \\
 &\lesssim
 \|f\|_{\mathcal{M}^{\Phi,\varphi_1}}.
\end{align*}
Simply replace
${\rm W}L^\Psi(B)$ with $L^\Psi(B)$
and
${\rm W}\mathcal{M}^{\Psi,\eta}({\mathbb R}^n)$
with
$\mathcal{M}^{\Psi,\eta}({\mathbb R}^n)$ for the strong estimate.

\medskip

{\it 2.}
We will now prove the second part.
We utilize \eqref{trpwesfr}.
By Lemma \ref{charOrlMor}, we have
\begin{align*}
\rho(r)&\lesssim |B(0,r)|^{-1}\|M_{\rho} \chi_{B(0,2r)}\|_{WL^{1}(B(0,r))} \lesssim \varphi_2(r)\|M_{\rho} \chi_{B(0,2r)}\|_{{\rm W}\mathcal{M}^{1,\varphi_{2}}}
\\
&\lesssim \varphi_2(r)\|\chi_{B(0,2r)}\|_{\mathcal{M}^{\Phi,\varphi_{1}}}\lesssim  \frac{\varphi_2(r)}{\varphi_1(r)}.
\end{align*}

\medskip

{\it 3.}
Since $\varphi_2$ is almost decreasing, \eqref{eq3.6.VZPotnoxFr} and \eqref{condSpanFr} are equaivalent.
Then the third statement of the theorem follows from the first and second parts of the theorem.
\end{proof}

\subsection{Adams-type result}
The following theorem gives necessary and sufficient conditions for Adams-type
boundedness of the operator $M_{\rho}$
from $\mathcal{M}^{\Phi,\varphi}({\mathbb R}^n)$ to $\mathcal{M}^{\Psi,\eta}({\mathbb R}^n)$.

Here we suppose that
$\rho$ is an increasing function such that
$r \in (0,\infty) \mapsto r^{-n}\rho(r) \in (0,\infty)$
is decreasing.

\begin{theorem}\label{thm:180314-11}
Let $\Phi$ be a Young function, and let $\varphi\in{\mathcal{G}}_{\Phi}$ be almost decreasing.
Assume that $\rho$ is increasing
and that $r\mapsto r^{-n}\rho(r)$ is decreasing.
Let $\beta\in(0,1)$, $\eta(t)\equiv\varphi(t)^{\beta}$ and $\Psi(t)\equiv\Phi(t^{1/\beta})$.
Then the condition
\begin{equation}\label{condAdamsfr}
\rho(t)\lesssim \varphi(t)^{\beta-1},
\end{equation}
is necessary and sufficient for the boundedness of $M_{\rho}$
from $\mathcal{M}^{\Phi,\varphi}({\mathbb R}^n)$ to ${\rm W}\mathcal{M}^{\Psi,\eta}({\mathbb R}^n)$.
Moreover, if $\Phi\in\nabla_2$, then the condition~\eqref{condAdamsfr}
is necessary and sufficient for the boundedness of $M_{\rho}$
from $\mathcal{M}^{\Phi,\varphi}({\mathbb R}^n)$ to $\mathcal{M}^{\Psi,\eta}({\mathbb R}^n)$.
\end{theorem}

As before, we start with an auxiliary pointwise estimate.
\begin{lemma}
Under the assumption of Theorem \ref{thm:180314-11}
including $(\ref{condAdamsfr})$,
there exists a positive constant $C$ such that,
for all $f\in\mathcal{M}^{\Phi,\varphi}({\mathbb R}^n)$ and all $x\in{\mathbb R}^n$,
\begin{equation}\label{poiwseestfrmax}
M_{\rho} f(x) \le C(Mf(x))^{\beta} \, \|f\|_{\mathcal{M}^{\Phi,\varphi}}^{1-\beta}.
\end{equation}
\end{lemma}

\begin{proof}
For arbitrary ball $B=B(x,r)$ we represent $f$ as
\begin{equation*}
f=f_1+f_2, \ \quad f_1=f\chi _{B},\quad
 f_2=f-f_1, \ \quad r>0,
\end{equation*}
so that
$$
M_\rho f(x)\le M_\rho f_1(x)+M_\rho f_2(x).
$$

Hence by Lemma~\ref{lemHold},
\begin{equation*}
\begin{split}
M_{\rho} f_2(x) & = \sup_{t>0}\frac{\rho(t)}{|B(x,t)|} \int_{B(x,t)\setminus (B(x,r))}|f(z)|d z
\\
& \le \, \sup_{t>r}\frac{\rho(t)}{|B(x,t)|} \int_{B(x,t)}|f(z)|d z
\\
&\lesssim \sup_{t>r} \rho(t) \, \Phi^{-1}(|B(x,t)|^{-1}) \, \|f\|_{L^{\Phi}(B(x,t))}\\
&\lesssim \|f\|_{\mathcal{M}^{\Phi,\varphi}} \sup_{t>r} \rho(t) \, \varphi(t).
\end{split}
\end{equation*}
Consequently by Lemma \ref{swkshr} 
we have
\begin{equation*}
M_\rho f(x) \lesssim
\rho(r) Mf(x)+\|f\|_{\mathcal{M}^{\Phi,\varphi}} \sup_{t>r} \rho(t) \, \varphi(t).
\end{equation*}

Thus, using the technique in \cite[p. 6492]{SawSugTan2} as before and \eqref{condAdamsfr} we obtain
\begin{align*}
M_{\rho} f(x) & \lesssim  \min \{ \varphi(r)^{\beta-1} Mf(x), \varphi(r)^{\beta} \|f\|_{\mathcal{M}^{\Phi,\varphi}}\}
\\
& \lesssim \sup\limits_{s>0} \min \{ s^{\beta-1} Mf(x), s^{\beta} \|f\|_{\mathcal{M}^{\Phi,\varphi}}\}
\\
& = (Mf(x))^{\beta} \, \|f\|_{\mathcal{M}^{\Phi,\varphi}}^{1-\beta},
\end{align*}
where we have used that the supremum is achieved when the minimum parts are balanced.
This shows \eqref{poiwseestfrmax}.
\end{proof}

We prove Theorem \ref{thm:180314-11}.

\begin{proof}[Proof of Theorem \ref{thm:180314-11}]
By using inequality \eqref{poiwseestfrmax} and Lemma \ref{lem:180314-6}
we have,
for all balls $B$,
$$
\|M_{\rho} f\|_{WL^{\Psi}(B)}
\lesssim \|(Mf)^{\beta}\|_{WL^{\Psi}(B)}\, \|f\|_{\mathcal{M}^{\Phi,\varphi}}^{1-\beta}
= \|Mf\|_{WL^{\Phi}(B)}^{\beta}\, \|f\|_{\mathcal{M}^{\Phi,\varphi}}^{1-\beta}.
$$
%
Consequently, by using the boundedness of the maximal operator $M$, we get
\begin{align*}
& \eta(r)^{-1}\Psi^{-1}(|B|^{-1})\|M_{\rho} f\|_{WL^{\Psi}(B)}
\lesssim
\eta(r)^{-1}\Psi^{-1}(|B|^{-1})
\|Mf\|_{WL^{\Phi}(B)}^{\beta}\, \|f\|_{\mathcal{M}^{\Phi,\varphi}}^{1-\beta}\\
&=
\left(\varphi(r)^{-1}\Phi^{-1}(|B|^{-1})
\|Mf\|_{WL^{\Phi}(B)}\right)^{\beta}\, \|f\|_{\mathcal{M}^{\Phi,\varphi}}^{1-\beta}
\lesssim
\|f\|_{\mathcal{M}^{\Phi,\varphi}}.
\end{align*}
By taking the supremum over all balls $B$, we get the desired result.
Moreover, if $\Phi\in\nabla_2$, then we have the strong type estimate.

We will now prove the necessity. We utilize
(\ref{trpwesfr}).
By Lemmas \ref{charorlc} and \ref{charOrlMor}, we have
\begin{align*}
\rho(r)&\lesssim \Psi^{-1}(r^{-n})\|M_{\rho} \chi_{B(0,2r)}\|_{WL^{\Psi}(B(0,r))}
\lesssim \eta(r)\|M_{\rho} \chi_{B(0,2r)}\|_{{\rm W}\mathcal{M}^{\Psi,\eta}}
\\
&\lesssim \eta(r)\|\chi_{B(0,2r)}\|_{\mathcal{M}^{\Phi,\varphi}}\lesssim  \frac{\eta(r)}{\varphi(r)}
\lesssim \varphi(r)^{\beta-1}.
\end{align*}
Then the proof is complete.
\end{proof}

\

{\bf Acknowledgements.} {The research of F. Deringoz was partially supported by the grant of Ahi Evran University Scientific Research Project (FEF.A4.18.019). 
The research of V. Guliyev was partially supported by the grant of 1st Azerbaijan--Russia Joint Grant Competition (Agreement number no. EIF-BGM-4-RFTF-1/2017-21/01/1) 
and by the Ministry of Education and Science of the Russian Federation (the Agreement No. 02.a03.21.0008).
Eiichi Nakai was supported
by Grant-in-Aid for Scientific Research (B), No. 15H03621, Japan Society for the Promotion of Science. 
Yoshihiro Sawano was supported by Grant-in-Aid for Scientific Research (C) (16K05209), the Japan Society for
the Promotion of Science and by People’s Friendship University of Russia.
}

\

\bigskip

\noindent
Fatih Deringoz \\
Department of Mathematics, Ahi Evran University, Kirsehir, Turkey \\
deringoz@hotmail.com
\\[3ex]
\noindent
Vagif S. Guliyev \\
Department of Mathematics, Dumlupinar University, Kutahya, Turkey \\
+Institute of Mathematics and Mechanics, Baku, Azerbaijan \\
+S.M. Nikolskii Institute of Mathematics at RUDN University, 
Moscow, Russia 117198 \\
vagif@guliyev.com
\\[3ex]
\noindent
Eiichi Nakai \\
Department of Mathematics, Ibaraki University, Mito, Ibaraki 310-8512, Japan \\
eiichi.nakai.math@vc.ibaraki.ac.jp  
\\[3ex]
\noindent
Yoshihiro Sawano \\
Department of Mathematics and Information Sciences, Tokyo Metropolitan University, Minami-Ohsawa 1-1, Hachioji, Tokyo, 192-0397, Japan \\
+S.M. Nikolskii Institute of Mathematics at RUDN University, Moscow, Russia 117198 \\
yoshihiro-sawano@celery.ocn.ne.jp 
\\[3ex]
\noindent
Minglei Shi \\
Department of Mathematics, Ibaraki University, Mito, Ibaraki 310-8512, Japan \\
18nd206l@vc.ibaraki.ac.jp, stfoursml@gmail.com

\end{document}